\documentclass{amsart}

\usepackage{graphicx}
\usepackage{tikz-cd}
\usepackage{amssymb}
\usepackage{upquote}
\usepackage{amsthm}
\usepackage[normalem]{ulem}
\usepackage{mathtools}
\usetikzlibrary{arrows}
\usepackage{color}
\usepackage{multirow}

\theoremstyle{plain}
\newtheorem{theorem}{Theorem}[section]
\newtheorem{lemma}[theorem]{Lemma}
\newtheorem{proposition}[theorem]{Proposition}
\newtheorem{corollary}[theorem]{Corollary}
\newtheorem*{corollary*}{Corollary}
\newtheorem*{theorem*}{Theorem}
\newtheorem{question}[theorem]{Question}

\theoremstyle{definition}
\newtheorem{definition}[theorem]{Definition}
\newtheorem{example}[theorem]{Example}

\theoremstyle{remark}
\newtheorem{remark}[theorem]{Remark}

\numberwithin{equation}{section}

\begin{document}
\def\Prod{\displaystyle \prod}
\def\Lim{\displaystyle \lim}
\def\Int{\displaystyle \int}
\def\Sum{\displaystyle \sum}
\def\Frac{\displaystyle \frac}
\def\Sup{\displaystyle \sup}
\def\Inf{\displaystyle \inf}
\def\Min{\displaystyle \otimes_{min}}
\def\Max{\displaystyle \otimes_{max}}
\def\R{\mathcal{R}}
\def\C{\mathbb{C}}
\def\D{\mathbb{D}}
\def\T{\mathbb{T}}
\def\N{\mathbb{N}}
\def\Q{\mathbb{Q}}
\def\F{\mathbb{F}}
\def\M{\mathbb{M}}
\def\Z{\mathbb{Z}}
\def\H{\mathcal{H}}
\def\A{\mathcal{A}}
\def\I{\mathcal{I}}
\def\x{\textbf{x}}
\def\B{\mathbb{B}}
\def\K{\mathcal{K}}
\def\cQ{\mathcal{Q}}
\def\MAX{\otimes_{max}}
\def\MIN{\otimes_{min}}

\title{The universal $C^*$-algebra of a contraction}

\author{Kristin Courtney}
\thanks{The research of the first-named author was partially supported by the Eric Nordgren Research Fellowship Fund at the University of New Hampshire.}
\address{Department of Mathematics, University of Virginia, 141 Cabell Drive, Kerchof Hall, P.O. Box 400137, Charlottesville, VA 22904-4137 USA}
\curraddr{Mathematical Institute, WWU M\"{u}nster, Einsteinstr. 62, 48149 M\"{u}nster, Germany}
\email{kcourtne@uni-muenster.de}

\author{David Sherman}
\address{Department of Mathematics, University of Virginia, 141 Cabell Drive, Kerchof Hall, P.O. Box 400137, Charlottesville, VA 22904-4137 USA}
\email{dsherman@virginia.edu}

\date{\today}

\begin{abstract}
We say that a contractive Hilbert space operator is universal if there is a natural surjection from its generated $C^*$-algebra to the $C^*$-algebra generated by any other contraction.  A universal contraction may be irreducible or a direct sum of (even nilpotent) matrices; we sharpen the latter fact and its proof in several ways, including von Neumann-type inequalities for noncommutative *-polynomials.  We also record properties of the unique $C^*$-algebra generated by a universal contraction, and we show that it can be used similarly to $C^*(\F_2)$ in various Kirchberg-like reformulations of Connes' Embedding Problem (some known, some new).  Finally we prove some analogous results for universal $C^*$-algebras of 
row contractions and universal Pythagorean $C^*$-algebras.





\end{abstract}

\maketitle

\tableofcontents

\section{Introduction}

For $T$ a contractive linear operator on a Hilbert space, let $C^*(T,I)$ denote the unital $C^*$-algebra generated by $T$ and the identity $I$. We say that $T$ is a \textit{universal contraction} if for any other contractive Hilbert space operator $S$ there is a unital *-homomorphism from $C^*(T,I)$ to $C^*(S,I)$ taking $T$ to $S$.  This is equivalent to requiring that for any noncommutative *-polynomial $q$, the norm of $q(T)$ is as large as it can be for a contraction.  (It should not be confused with ``universal (model) operators" as introduced by Rota \cite{R1959,R1960}, although we explain a connection in Remark \ref{R:model}.)

One goal of this paper is to describe universal contractions themselves.  On a given separable infinite-dimensional Hilbert space, universal contractions comprise a single approximate unitary equivalence class that is strong* dense in the set of all contractions (Propositions \ref{P:aue} and \ref{P:s*}).
But there are qualitative differences among them: a single universal contraction can be irreducible (Theorem \ref{T:prim}) or a direct sum of nilpotent matrices (Theorem \ref{nilprep}).  The latter result is known, but the existing proof is quite intricate.  In Section \ref{nilp} we give two alternative proofs, and in Section \ref{row} we give a generalization.

If $T_1$ and $T_2$ are universal contractions, then $C^*(T_1,I)$ is *-isomorphic to $C^*(T_2,I)$ via the map sending $T_1$ to $T_2$. Thus we may call this $C^*$-algebra (equipped with its distinguished generator) the \textit{universal unital $C^*$-algebra of a contraction}. We will denote it here by $\A$, and we will denote the non-unital $C^*$-algebra $C^*(T_1) \simeq C^*(T_2)$ by $\A_0$. (It is indeed non-unital; see Remark \ref{nounit}.) The familiar reader will notice that $\A$ can be identified with the universal unital $C^*$-algebra associated to the operator space $\C$, as in \cite{Pis03}, and the maximal $C^*$-dilation of the disk algebra, as in \cite{BD}. In the usual generator-relation notation for universal $C^*$-algebras, we write
$$\A_0 = C^*\langle x : \: \|x\| \leq 1 \rangle.$$
Adding a unit to the generating set and relations above gives us the unitization of $\A_0$, which is exactly $\A$.  
The two are essentially interchangeable in this paper, but we often make use of the unit and so choose $\A$ for the main discussion.

A $C^*$-algebra is called \textit{residually finite dimensional} (RFD) if it has a separating family of finite-dimensional representations. The fact that a universal contraction can be a direct sum of matrices says that $\A$ is RFD and gives us a *-polynomial analogue of von Neumann's inequality: the maximal norm of any noncommutative *-polynomial whose entry ranges over contractive Hilbert space operators can be determined by considering only contractive (even nilpotent) matrices.  In Section \ref{max} we sharpen this by showing that the maximal norm is actually achieved at a contractive matrix whose dimension depends only on the degree of the *-polynomial (Theorem \ref{FNT}).  But not all elements of $\A$ attain their norm in a finite-dimensional representation (Proposition \ref{P:AF}).  In Section \ref{Projectivity} we discuss the projectivity of $\A$ as a unital $C^*$-algebra, which is stronger than RFD \cite[Lemma 8.1.4]{Lor97} and implies that $\A$ has trivial topological invariants such as K-theory and shape theory.  Note that any separable RFD algebra has a faithful tracial state, so $\A$ is a finite $C^*$-algebra.

On the other hand, the algebra $\A$ is the mother of all singly-generated unital $C^*$-algebras, so it also has various ``largeness" properties.  Because it surjects onto a non-exact algebra, it is not exact and \textit{a fortiori} not nuclear (\cite[Corollary 1]{BN12} or \cite[Section 6.1]{BD}, although ``$\mathbb{K} \otimes B$" in the latter should be unitized).  Because it surjects onto a non-type I algebra (e.g., a UHF algebra, singly generated by \cite{Top68}), it is not type I.  This leads to comparisons with other large $C^*$-algebras, like the universal $C^*$-algebra of a partial isometry, to which $\A$ is Morita equivalent \cite{BN12}, and full $C^*$-algebras of free groups. In fact several of the results for $\A$ in Section 2 are patterned on corresponding results for $C^*(\F_2)$ obtained in the beautiful 1980 paper of Choi \cite{Cho80}.  Note that 
$C^*(\F_2)$ is not singly generated as a unital $C^*$-algebra, so it cannot be a quotient of $\A$.  However, $\A$ is a quotient of $C^*(\F_2)$, and they embed into each other relatively weakly injectively (Theorem \ref{rwithm}). Using Theorem \ref{rwithm} and results from Kirchberg's seminal work \cite{Kir93}, we give in Theorem \ref{QWEPlem} several characterizations of Connes' Embedding Problem in terms of $\A$ (some of these were deduced differently in \cite{Pis03,BD}). 

In Section \ref{remarks} we adapt some of our arguments and results for $\A$ to other $C^*$-algebras generated by multiple universal contractions. Theorem \ref{nilrow} generalizes Theorem \ref{nilprep} to 
row contractions. This leads to a Popescu-von Neumann inequality for noncommutative *-polynomials on 
row contractions (Corollary \ref{P-vN}). Section \ref{pythagorean} establishes that the universal Pythagorean $C^*$-algebras from \cite{BJ} are RFD. 

Our notation throughout the paper is fairly standard. For an operator $T\in B(\H)$, we denote its spectrum by $\sigma(T)$ and its essential spectrum by $\sigma_e(T)$. We denote the open unit disk and unit circle of $\C$ by $\D$ and $\T$, respectively.  We write $\M_n$ for the complex $n \times n$ matrices and freely associate it with $B(\C^n)$.

After this article was completed, the authors became aware of unpublished work from 1989 by Froelich and Salas along similar themes.  Upon request Prof. Salas graciously shared their manuscript; most of the overlap concerns results in Section 2.  We decided to make no changes to the present article except mention of the reference \cite{CH}.

The authors are grateful to David Blecher, Scott Atkinson, and an anonymous referee for useful comments on a draft of this article, and to Don Hadwin for many valuable discussions and perspectives. Some of this material is taken from the first-named author's 2018 PhD dissertation at the University of Virginia \cite{Cou18}. 

\section{Universal contraction operators} \label{S:op}

We start with an easy observation.

\begin{proposition} \label{P:norm}

The following are equivalent for a contractive operator $T$ on a Hilbert space.

\begin{enumerate}
\item The operator $T$ is a universal contraction; i.e., for any other contractive Hilbert space operator $S$, the assignment $T \mapsto S$ induces a *-homomorphism from $C^*(T,I)$ to $C^*(S,I)$.
\item For any noncommutative *-polynomial $q$,
\begin{equation} \label{E:norm}
\|q(T)\| = \sup_{S} \|q(S)\|,
\end{equation}
where the supremum is taken over all contractive Hilbert space operators.
\end{enumerate}
\end{proposition}

\begin{proof}
We have that (1) implies (2) because *-homomorphisms are contractive, and (2) implies (1) because the assignment $q(T) \mapsto q(S)$ densely defines a continuous *-homomorphism between the $C^*$-algebras.  
\end{proof}

How can we produce universal contractions?  For the reader versed in universal $C^*$-algebras defined by generators and relations, it is clear that $\A$ is a separable nonzero algebra (see for instance \cite[Section 3.1]{Lor97}).  So it has a faithful representation on $\ell^2$, and the image of its distinguished generator is a separably-acting universal contraction.

The next proposition exhibits a universal contraction that is a little more concrete.  It is ``known to the experts" and seems to have been first mentioned as a tool in the proof of \cite[Theorem 5.1]{Had78}, but we lack a reference that provides the details.  Since it plays an important role in this paper, for the benefit of the reader we give a short argument using only basic operator theory.  (It also can be proved by appealing to the projectivity of $\A$; see the text after Remark \ref{nounit}.)


\begin{proposition} \label{P:hadwin} 

${}$

\begin{enumerate}
    \item The supremum in \eqref{E:norm} is unchanged if $S$ ranges only over matrix contractions.  
    \item Let $T$ be any direct sum of contractive matrices such that the $n \times n$ summands are dense in $(\M_n)_{\leq 1}$ for each $n$.  Then $T$ is a universal contraction.
    \item The algebra $\A$ is RFD.
\end{enumerate}
\end{proposition}

\begin{proof}
All of these statements effectively say the same thing, so it suffices to prove (1).  (To see that $\A$ is RFD, note that restriction to the matrix summands of the operator in (2) gives a separating family of representations.)

Let $S$ be a contractive operator on the Hilbert space $\H$.  For $F$ a finite-dimensional subspace of $\H$, let $P_F$ be the projection onto $F$.  Then the net $\{P_F S P_F\}$, ordered by inclusion of the subspaces $F$, converges strong* to $S$.  For any noncommutative *-polynomial $q$, we also have $q(P_F S P_F) \to q(S)$ strong*.  By strong* lower semicontinuity of the norm, we have
$$\|q(S)\| = \|s^*-\lim q(P_F S P_F)\| \leq \liminf \|q(P_F S P_F)\|.$$
Now $P_F S P_F$ is the direct sum of a matrix and some multiple of the $1 \times 1$ zero operator.  It follows that 
$\|q(S)\|$ is not more than the maximal norm of $q$ evaluated on contractive matrices.  (In fact it is a maximum; see Theorem \ref{FNT} below.)
\end{proof}



In Section \ref{nilp} we give proofs and variations for the much harder fact, basically due to Herrero, that a universal contraction can be built as a direct sum of \textit{nilpotent} matrices.

So must a universal contraction be a direct sum of some sort?  No, we show below in Theorem \ref{T:prim} that there are irreducible universal contractions.  This is equivalent to proving that $\A$ is \textit{primitive}, meaning that it has a faithful irreducible representation.  Our argument mimics Choi's proof of primitivity for $C^*(\F_2)$ \cite[Theorem 6]{Cho80}, with additional reliance on the fact that $\A$ is RFD.


\begin{proposition} \label{P:compact}
The algebra $\A$ contains no nontrivial projections.  If $\pi$ is a faithful representation of $\A$, then $\pi(\A)$ contains no nonzero compact operators.
\end{proposition}

\begin{proof}
This can be proved in the same way as \cite[Theorem 1 and Corollary 2]{Cho80}. 
\end{proof}

\begin{lemma} \label{L:perturb}
Let $T$ be a universal contraction and $K$ be a compact operator.  If $T+K$ is a contraction, then it is a universal contraction.
\end{lemma}

\begin{proof}
Let $T$ and $K$ be as stated, let $q$ be a noncommutative *-polynomial, and let $\pi:B(\H)\to B(\H)/K(\H)$ be the Calkin map.  It follows from Proposition \ref{P:compact} that the restriction of $\pi$ to $C^*(T,I)$ is isometric. Since $\pi(q(T))=q(\pi(T)) = q(\pi(T+K)) = \pi(q(T+K))$, we have
$$\|q(T+K)\| \geq \|\pi(q(T+K))\| = \|\pi(q(T))\| = \|q(T)\|$$
and are done by Proposition \ref{P:norm}.
\end{proof}

\begin{theorem} \label{T:prim}
The algebra $\A$ is primitive.  Equivalently, there exist irreducible universal contractions.
\end{theorem}

\begin{proof}

Choi used the following lemma, based on techniques of Radjavi-Rosenthal \cite[Theorem 7.10 and Theorem 8.30]{RR73}.  Let $A$ and $B$ be operators on the same Hilbert space, expressed as matrices with respect to a basis.  If $A$ is diagonal with distinct entries, and $B$ has no zero entries in its first column, then $A$ and $B$ have no common nontrivial reducing subspace.  We will find a universal contraction whose real and imaginary parts can be taken as $A$ and $B$ above.  Since a projection commutes with an operator if and only if it commutes with the real and imaginary parts, the contraction will have no nontrivial reducing subspaces. 

As in Proposition \ref{P:hadwin}, construct a universal contraction $T = \oplus_{j=1}^\infty T_j$ on $\H=\oplus_{j=1}^\infty \C^{n_j}$ as a direct sum of matrices $T_j\in \mathbb{M}_{n_j}$.  We may choose the matrices $T_j$ to be strict contractions, and we may choose the ordering so that $\|T_j\| < 1 - \frac1j$. By taking unitary conjugates we may also assume that the real part of each $T_j$ is diagonal.

For each $j \geq 1$, perturb the diagonal entries of $\text{Re }T_j$ by less than $\frac{1}{2j}$ so that the entries are distinct from each other and all diagonal entries of $T_{i}$ for $i<j$. 
This perturbs $T$ to $T'$ so that $\text{Re }T'$ is diagonal with distinct entries, and the corresponding summands have $\|T'_j\| < 1 - \frac{1}{2j}$. Because this is a compact perturbation, $T'$ is still a universal contraction by Lemma \ref{L:perturb}.

Now perturb $T'$ to $T''$ as follows. We only change the first column and row so that these are everywhere nonzero in $\text{Im }T''$. We go one block at a time, taking advantage of the little bit of norm wiggle room. In the first block $T'_1$, change all entries in the first column and row by some small identical pure imaginary amount $\lambda_1$, making them all nonzero and keeping $\|T''_1\| < 1 - \frac13$. Then change the part of the first column and row of $T'$ corresponding to the second block - these are all zero in $T'$ - by a small identical pure imaginary amount $\lambda_2$, making them all nonzero and preserving that the submatrix of $T''$ corresponding to the first two blocks has norm $< 1 - \frac16$. Continue in this way for all blocks $T'_j$. Visually, we are perturbing $T'$ by the operator 
$$R=\begin{pmatrix} \lambda_1 & \lambda_1 & \cdots & \lambda_1 & \lambda_2 & \cdots \lambda_n & \cdots\\
\lambda_1 &&&&&&&\\
\vdots &&&&&&&\\
\lambda_1 &&&&&&&\\
\lambda_2 &&&&\text{\huge0}&&&\\
\vdots &&&&&&&\\
\lambda_n &&&&&&&\\
\vdots\\
\end{pmatrix}$$
where each $\lambda_j\in i\mathbb{R}$ is repeated $n_j$ times.

The submatrix of $T''=T'+R$ corresponding to the initial string of $n$ blocks has norm $< 1 - \frac{1}{3n}$.  The operator $T''$ is a strong limit of these submatrices (considered as infinite-dimensional operators by filling in the rest of the matrix with zeroes), so $T''$ is also a contraction. And because we have only changed the first row and column to go from $T'$ to $T''$, it is a compact perturbation, and $T''$ is still a universal contraction by Lemma \ref{L:perturb}.

Notice that the matrix for $\text{Re }T'' = \text{Re }T'$ is diagonal with distinct entries, and the matrix for $\text{Im }T''$ has nonzero entries in the first column.  By the result mentioned at the beginning of the proof, we are done.
\end{proof}

It is apparent from the preceding results that separably-acting universal contractions need not be unitarily equivalent, but the truth is not so far from that.  Recall that two operators are said to be \textit{approximately unitarily equivalent} (a.u.e.), denoted here $\sim_a$, if one is the norm limit of unitary conjugates of the other.  Two representations of the same $C^*$-algebra are a.u.e. if one representation is the point-norm limit of unitary conjugates of the other.

Parts of the next two propositions are noted or implied in Hadwin's work on Voiculescu's noncommutative Weyl-von Neumann theorem \cite{Had77,Had78,Had81} and the subsequent discussion in \cite[Section 7]{CH}.  We include some citations but make all the reasoning explicit here.

\begin{proposition} \label{P:aue}
Separably-acting universal contractions form a single a.u.e. class.
\end{proposition}

\begin{proof}
Let $T_1$ and $T_2$ be separably-acting universal contractions.  By sending the distinguished generator of $\A$ to $T_j$, we obtain two faithful representations of $\A$ whose ranges contain no nontrivial compact operators by Proposition \ref{P:compact}.  It then follows from Voiculescu's noncommutative Weyl-von Neumann theorem \cite[Corollary 1.4]{Voi76} that the representations are a.u.e., which means that $T_1 \sim_a T_2$.

If $S_1 \sim_a S_2$, then for any noncommutative *-polynomial $q$ we have $q(S_1) \sim_a q(S_2)$ and $\|q(S_1)\| = \|q(S_2)\|$.  So any operator a.u.e. to a universal contraction must also be one by Proposition \ref{P:norm}.
\end{proof}

\begin{proposition} \label{P:s*}
For $T \in B(\ell^2)_{\leq 1}$, the following are equivalent.
\begin{enumerate}
    \item $T$ is a universal contraction.
    \item $T \sim_a (T \oplus S)$ for every $S \in B(\ell^2)_{\leq 1}$.
    \item The strong* closure of the unitary orbit of $T$ is $B(\ell^2)_{\leq 1}$.
\end{enumerate}
\end{proposition}

\begin{proof}
(1) $\iff$ (2): We first show that the operators on $\ell^2$ satisfying the condition in (2) form a single a.u.e. class.  If $T_1, T_2$ satisfy the condition, then $$T_1 \sim_a T_1 \oplus T_2 \sim_a T_2.$$  And if $T \sim_a T_1$ and $S \in B(\ell^2)_{\leq 1}$, then $$T \oplus S \sim_a T_1 \oplus S \sim_a T_1 \sim_a T.$$

The conclusion then follows from Proposition \ref{P:aue} and Hadwin's observation \cite[Example 7.3(1)]{Had81} that a universal contraction constructed in Proposition \ref{P:hadwin}(2) satisfies the condition in (2).

(2) $\Rightarrow$ (3): Hadwin shows in \cite[Corollary 3.4 and Theorem 4.3]{Had77} that for \textit{any} $T \in B(\ell^2)$, the strong* closure of the unitary orbit of $T$ is the set of operators that are summands of operators a.u.e. to $T$.

(3) $\Rightarrow$ (1): Assume (3), and let $S$ be any contraction on $\ell^2$.  Let $\{U_j\}$ be unitaries with $U_j^* T U_j \to S$ strong*, and let $q$ be a noncommutative *-polynomial.  Then
\begin{align*}
\|q(S)\| &= \|q(s^*-\lim (U_j^* T U_j))\| = \|s^*-\lim q(U_j^*TU_j)\| \\ &=\|s^*-\lim U_j^*q(T)U_j \| \leq \liminf \|U_j^* q(T) U_j\| = \|q(T)\|.
\end{align*}
Thus $T$ is universal by Proposition \ref{P:norm}.
\end{proof}

\begin{remark} \label{R:model}
``Universality" has various meanings for operators, some connected to our situation and some not.  In the context of operator ideals it is a factorization property (e.g., \cite{O16}) -- this is quite different from the present paper.  

In the definition perhaps most familiar to an operator theorist, a Hilbert space operator $T$ is said to be universal, or be a \textit{universal model}, if every separably-acting operator can be scaled to become similar to the restriction of $T$ to some invariant subspace.  See \cite[Chapter 6]{K97} for a discussion of the area.  
The terminology originates with Rota \cite{R1959}, who proved in \cite{R1960} that $B$, the backward shift with infinite multiplicity, is a universal model in a slightly more precise sense: any operator with spectral radius $< 1$ is similar to the restriction of $B$ to an invariant subspace.  Proposition \ref{P:s*}(2) characterizes the universality of $T$ (as defined in this paper) analogously: any operator with norm $\leq 1$ is unitarily equivalent to the restriction of an operator a.u.e. to $T$ to a reducing subspace.
\end{remark}

Operator theorists may be tempted to look for universal contractions among the weighted shifts, but that cannot succeed because weighted shifts are \textit{centered} \cite{MM74}; they satisfy relations saying that the family $\{T^mT^{*m}, T^{*n}T^n : \, m,n \geq 0\}$ is commutative.  In fact the authors know of no explicit ``naturally-occurring" universal contractions.  Our constructions all rely in some way on the direct sum of a dense set.

\section{Von Neumann-type inequalities for noncommutative *-polynomials}\label{max}


Let $M_z$ denote multiplication by $z$ on $L^2(\T,m)$, and let $p$ be a polynomial.  From compactness, the maximum modulus principle, and spectral theory, we have 
\begin{equation} \label{E:vN}
[\sup \text{ or } \max]_{\bar{\D}} |p(\lambda)| = [\sup \text{ or } \max]_\T |p(\lambda)|  = \|p(M_z)\|.
\end{equation}
The celebrated \textit{von Neumann inequality} \cite{vN} is the fact that for any contraction $S$, the norm of $p(S)$ is dominated by the quantity in \eqref{E:vN}.

We may then say that $M_z$ is a universal contraction for polynomials.  (In fact a contraction has this property if and only if its spectrum contains $\T$.)\footnote{The backward implication holds because the operator norm is not smaller than the spectral radius.  For the forward implication, let $\lambda \in \T$ and $T$ be a contraction with this property.  The spectral radius of $(T+\lambda)$ is $\lim \|(T+\lambda)^n\|^{1/n} = \lim \|(z+\lambda)^n\|_{C(\T)}^{1/n} = 2$, which forces $\lambda \in \text{sp}(T)$.}  Equivalently, $M_z$ generates the universal unital operator algebra\footnote{Here an \textit{operator algebra} is a (not necessarily self-adjoint) norm-closed subalgebra of some $B(\H)$, or any matrix-normed Banach algebra that can be represented as such.} of a contraction, which is isometrically isomorphic as a Banach algebra to the disk algebra $A(\D)$ via the map densely defined by $p \leftrightarrow p(M_z)$ \cite[Example 2.2]{BP91}.  
All of the foregoing is still true if we replace $p$ with any $f \in A(\D)$, and so von Neumann's inequality opens the door for an analytic functional calculus for all contractions.  (For any contraction $S$, $f(S)$ is well-defined by von Neumann's inequality as the limit of $p_j(S)$, where $\{p_j\}$ is any sequence of polynomials converging to $f$ in $A(\D)$.)

In summary, von Neumann's inequality says that there are sufficiently many one-dimensional representations of the universal unital operator algebra of a contraction to determine the norm of a polynomial in the generator (and thus we can identify the algebra with $A(\D)$ and the one-dimensional representations with $\bar{\D}$).  In fact the norm is achieved.  The smaller set $\T$ suffices to determine the norm, which is still achieved.  For any element of $A(\D)$ the norm is achieved at a one-dimensional representation.

We pursue analogues of every clause in the preceding paragraph, in the context of noncommutative *-polynomials and $C^*$-algebras, and call these analogues (when valid) ``von Neumann-type inequalities."

Let $x$ be the canonical generator of $\A$, $S$ any contraction, and $q$ a noncommutative *-polynomial.  By definition we have that $\|q(S)\| \leq \|q(x)\|$.  We typically cannot determine $\|q(x)\|$ with the one-dimensional representations of $\A$ (i.e., replacing $x$ with scalars): consider $q(z) = z^*z - zz^*$.  But the finite-dimensional representations of $\A$ do suffice (i.e., replacing $x$ with contractive matrices).  We already proved this in Proposition \ref{P:hadwin}(1), which may thus be considered a von Neumann-type inequality:
\begin{equation} \label{E:*vN}
\|q(S)\| \leq \|q(x)\| = \sup_{M \text{ contractive matrix}} \|q(M)\|.
\end{equation}

Here are our questions.  Note that there is no compactness available.

\begin{enumerate}
\item Is the supremum in \eqref{E:*vN} achieved?
\item Can we replace the contractive matrices with a smaller natural set?  And will the supremum still be achieved?
\item Since noncommutative *-polynomials in $x$ form a dense set in $\A$, it follows from the above that for any $a \in \A$, $\|a\| = \sup_\pi  \|\pi(a)\|$, as $\pi$ ranges over the finite-dimensional representations of $\A$.  Is this supremum achieved?
\end{enumerate}

Here are the answers, the first and third 
of which we proceed to show in the remainder of this section.
\begin{enumerate}
\item Yes (Theorem \ref{FNT}).
\item Contractive nilpotent matrices suffice, but the supremum is not achieved in general (Theorem \ref{nilprep} and Remark \ref{notnilp}).
\item No, it is not achieved for all elements of $\A$, but it is for some elements that are not of the form $q(x)$ (Proposition \ref{P:AF}).
\end{enumerate}

In \cite[Theorem 3.2]{CS} it was shown that a $C^*$-algebra is RFD exactly when it has a dense subset of elements that attain their norm under a finite-dimensional representation of the algebra. In some cases, this dense subset contains all noncommutative *-polynomials in the standard generators. For instance, Fritz, Netzer, and Thom prove in \cite[Lemma 2.7]{FNT} that every element in $\C\F_n$ attains its norm under some finite-dimensional representation of $C^*(\F_n)$, where $\F_n$ is a free group on $n \leq \infty$ generators. 
Our proof of Theorem \ref{FNT} below is a simplified version of the proof of \cite[Lemma 2.7]{FNT}, which itself is an adaptation of Choi's argument in \cite[Theorem 7]{Cho80} that $C^*(\F_n)$ is RFD.

The \textit{degree} of a noncommutative *-polynomial is the length of its longest monomial.


\begin{theorem}\label{FNT}
Let $q$ be a noncommutative *-polynomial of degree $d$, and let $x$ be the canonical generator of $\A$. Then 
\begin{equation} \label{E:Mn}
\|q(x)\|=\max\{\|q(M)\|:M \in \mathbb{M}_{2^{d+1}}, \|M\|\leq 1\}.
\end{equation}
\end{theorem}

\begin{proof}
 Let $\varphi$ be a state on $\A$ with $\varphi(q(x)^*q(x)) = \|q(x)\|^2$, and let $(\pi, \H, \xi)$ be the associated GNS representation.  Then $\|q(x)\|=\|\pi(q(x))\xi\|$. Define $$\H_0=\text{span}\{\pi(g(x))\xi: g \ \text{is a *-monomial of length}\ \ell(g)\leq d\}.$$
Note that dim$(\H_0)\leq 2^{d+1}$. Let $V$ be the inclusion of $\H_0$ in $\H$, so that $VV^*$ is the projection in $B(\H)$ onto $\H_0$.  By construction 
\begin{equation} \label{E:H}
V q(V^*\pi(x)V) V^*\xi = q(\pi(x))\xi,
\end{equation}
since $VV^*$ acts as the identity everywhere in the expansion of the left-hand side.

We can think of $V^*\pi(x)V \in B(\H_0)$ as a contractive matrix of size $\leq 2^{d+1}$, so the assignment $x\mapsto V^*\pi(x)V$ induces a unital *-homomorphism $\pi_0:\A\to B(\H_0)$.  Compute
\begin{align*}
\|q(\pi_0(x))(V^*\xi)\| & \geq  \|Vq(\pi_0(x))V^*\xi\| \\ &=  \|Vq(V^*\pi(x)V)V^* \xi\| = \|q(\pi(x))\xi\| = \|\pi(q(x))\xi\|=\|q(x)\|.
\end{align*}
Thus the norm of $q(\pi_0(x))$ must be $\|q(x)\|$ (it cannot be bigger).  Finally note that $B(\H_0)$ can be (not necessarily unitally) included in $\M_{2^{d+1}}$.
\end{proof}

The reader may wonder if $\A$ has a separating family of representations of bounded (finite) dimension, so that the maximum in \eqref{E:Mn} can be taken in some fixed $\M_N$ for all $q$.  (For polynomials $N=1$ works!)  This is not so: for a $C^*$-algebra, the existence of a separating family of representations of dimension $\leq N$ implies that all irreducible representations have dimension $\leq N$ (see \cite[Proposition 3.6.3(i)]{Dix77}). 
%
Actually, it follows from \cite[Theorem 5.1]{CS} that there are elements of $\A$ whose maximal norm in $n$-dimensional representations grows according to any prescribed finite pattern in $n$.  But the methods in \cite{CS} are nonconstructive, so that one cannot exhibit such elements, and they are unlikely to be noncommutative *-polynomials in the generator.  As a complement to these results we display here explicit noncommutative *-polynomials that achieve their maximal norm at any prescribed dimension and no lower.

\begin{example}\label{E:attain}
For any $n\geq1$, consider the noncommutative *-polynomial
$$q_n(z) = (z^*z + z^{*2} z^2 + ... + z^{*n} z^n) + (1- zz^*).$$
We claim that the maximal norm of $q_n(M)$, where $M$ ranges over all contractions, is attained in $\M_{n+1}$ but no smaller matrix algebra.

We have that $q_n(M)$ is a sum of $n+1$ positive contractions, so its norm cannot be more than $n+1$.  Taking $S_n$ to be the forward shift in $\M_{n+1}$, $q_n(S_n)$ sends the first basis element $e_1$ to $(n+1) e_1$, so it has the maximal norm $n+1$.

Now suppose that $M$ is a matrix such that the positive matrix $q_n(M)$ has norm $n+1$.  Let $v$ be a unit eigenvector for $q_n(M)$ for the value $n+1$.  It follows that $v$ is fixed by all of $M^*M, \dots, M^{*n} M^n$ and annihilated by $MM^*$; this means that $M^*v = 0$ and $M^j v$ is a unit vector for $j=1,\dots n$.  We have $Mv \in \text{ran} (M) \perp \ker (M^*) \ni v$.  We also have $M^2 v \in \text{ran} (M^2) \perp \ker (M^{*2}) \ni v, Mv$.  (For the last membership, note $M^{*2}(Mv) = M^*(M^*M v) = M^*v = 0$.)  Continuing, we get that $v, Mv, M^2 v, \dots, M^n v$ form an orthonormal set of $(n+1)$ vectors on which $M$ acts as the forward shift.  Thus $M$ is a matrix of size at least $(n+1) \times (n+1)$.
\end{example}


\smallskip

Let $\A_{\mathcal{F}}$ denote the set of elements in $\A$ that achieve their norm under some finite-dimensional representation, i.e., $a\in \A_{\mathcal{F}}$ if there exists a finite-dimensional representation $\pi$ of $\A$ such that $\|a\|=\|\pi(a)\|$. Then Theorem \ref{FNT} says that the set $\mathcal{P}_*$ of noncommutative *-polynomials in $x$ is contained in $\A_{\mathcal{F}}$. It is natural to ask whether every element in $\A$ attains its norm under some finite-dimensional representation, i.e., whether or not $\A_{\mathcal{F}}$ is the entire space. The first-named author and T. Shulman explored this question for general $C^*$-algebras in \cite{CS}. As the next proposition shows, not every element of $\A$ achieves its norm under a finite-dimensional representation, but there are some elements in $\A\backslash \mathcal{P}_*$ that do. 

\begin{proposition} \label{P:AF}
 Let $\A_{\mathcal{F}}$ be defined as above. Then $$\mathcal{P}_*\subsetneq \A_{\mathcal{F}}\subsetneq \A.$$
\end{proposition}

\begin{proof}
By \cite[Theorem 4.4]{CS}, we know that $\A_{\mathcal{F}}\subsetneq \A$ if and only if $\A$ has a simple, infinite-dimensional AF subquotient.
The CAR algebra $\mathbb{M}_{2^\infty}$ is a simple, unital, infinite-dimensional AF $C^*$-algebra, which is singly generated by \cite{Top68}. Hence, it 
is isomorphic to a quotient of $\A$, and so $$\A_{\mathcal{F}}\subsetneq \A.$$

It also follows from \cite[Theorem 5.5]{CS} that $\A_{\mathcal{F}}$ is not closed under addition (or multiplication) 
and hence cannot equal $\mathcal{P}_*$, i.e.,
$${} \qquad \qquad \qquad 
    \mathcal{P}_*\subsetneq \A_{\mathcal{F}}. 
\qquad \qquad \qedhere$$
\end{proof}


The proof of Proposition \ref{P:AF} is again nonconstructive.  A specific element of $\A_\mathcal{F} \setminus \mathcal{P}_*$ is $e^{x^*x}$, which attains its norm of $e$ in the one-dimensional representation of $\A$ sending $x$ to 1.  More interestingly, we can use Example \ref{E:attain} to exhibit an element of $\A \setminus \A_\mathcal{F}$.

\begin{example}\label{E:AF}
For any contractive element $y$ of a $C^*$-algebra, consider the norm convergent series
$$q(y)=\sum_{n=1}^\infty 2^{-n} \frac{q_n(y)}{n+1},$$
where the $q_n$ are as in Example \ref{E:attain}.  Since $\|q_n(x)\| = n+1$, we have $\|q(x)\| \leq 1$.  Taking $S$ to be the unilateral shift on $\ell^2$ and $e_1$ the first standard basis vector, we also have $\|q(S)\| \geq \|q(S)e_1\|= \|e_1\| = 1$.  Thus $\|q(x)\|=1$.  But it follows from Example \ref{E:attain} that $\|q(M)\|<1$ for any finite-dimensional contraction $M$.
\end{example} 


\section{Universal $C^*$-algebras and projectivity}\label{Projectivity}
In this section, we consider properties of $\A$ as a universal $C^*$-algebra. Because other universal $C^*$-algebras will soon make appearances, now is a good time to introduce them formally. 
We will forgo constructions and refer the reader to those given in \cite[Chapter 3]{Lor97} or \cite[Section 1]{Bla85}. 
Though they can be defined in more generality, we will restrict ourselves to separable universal $C^*$-algebras defined by imposing norm constraints and noncommutative *-polynomial relations on the generators. 


Let $\mathcal{G}=\{x_1,...,x_n\}$ be a finite set and $\mathcal{R}$ a finite set of relations of the form 
\begin{itemize}
    \item $r(x_1,...,x_n)\geq 0$, or
    \item $\|s(x_1,...,x_n)\|\leq C$ for some $C\geq 0$
\end{itemize}
where $r$ and $s$ are noncommutative *-polynomials. To fend off existence issues, we require that the relations in $\mathcal{R}$ enforce norm bounds on the elements of $\mathcal{G}$ and that there exists some tuple $(A_1,...,A_n)$ of operators on some Hilbert space that satisfy $\R$. 
The universal $C^*$-algebra with generators $\mathcal{G}$ subject to relations $\R$ is denoted by $C^*\langle \mathcal{G}: \R\rangle$ and has the following defining universal property: given Hilbert space operators $T_1,...,T_n$ satisfying $\R$, the assignments $x_i\mapsto T_i$, $1\leq i\leq n$, induce a surjective *-homomorphism $C^*\langle \mathcal{G}: \R\rangle\to C^*(T_1,...,T_n)$. 

If $C^*\langle \mathcal{G}:\R\rangle$ has a unit, then the *-homomorphism is assumed to be unital. If it is not unital, then we denote its unitization by $C^*_u\langle \mathcal{G}: \R\rangle$; we also adopt the convention of identifying a unital $C^*$-algebra with its unitization. Note that unitizing a universal $C^*$-algebra is the same as adding a unit to the generating set, i.e., $$C^*\langle \{y_1,...,y_n\}\cup\{I\}: \R\cup\{I=I^*=I^2, y_iI=y_i=Iy_i, 1\leq i\leq n\}\rangle.$$ To see this, we faithfully represent $$C^*\langle \{y_1,...,y_n\}\cup\{I\}: \R\cup\{I=I^*=I^2, y_iI=y_i=Iy_i, 1\leq i\leq n\}\rangle$$ as $C^*(I_{\mathcal{H}},Y_1,...,Y_n)$ on some Hilbert space $\H$. By universality, we have a surjective *-homomorphism $C^*\langle \mathcal{G}: \R\rangle\to C^*(Y_1,...,Y_n)$ induced by sending $x_i\mapsto Y_i$. This extends to a surjective unital *-homomophism $C^*_u\langle \mathcal{G}: \R\rangle\to C^*(I_{\mathcal{H}},Y_1,...,Y_n)$, which is injective by the universality of the latter algebra. In particular, this means adding a unit to the generators of $\A_0$ and unitizing $\A_0$ give the same algebra, namely $\A$. (We shall see in Remark \ref{nounit} why $\A_0$ cannot have a unit.) As a convention, we say the unitization of the zero $C^*$-algebra is $\C$. 



Now we are ready to change gears and introduce one of the most important properties of $\A$. 
We call a $C^*$-algebra $A$ {\it projective} if given any $C^*$-algebra $B$ with closed two-sided ideal $I$ and quotient map $\pi:B\to B/I$, any *-homomorphism $\phi:A\to B/I$ lifts to a *-homomorphsim $\psi:A\to B$ so that $\phi=\pi\circ \psi$. 
In other words, a projective $C^*$-algebra is a projective object in the category of $C^*$-algebras with *-homomorphisms. 
By \cite[Proposition 2.5]{Bla85}, a $C^*$-algebra is projective if and only if its unitization is projective in the category of unital $C^*$-algebras with unital *-homomorphisms. For simplicity, if we call a unital $C^*$-algebra projective, we mean in the category of unital $C^*$-algebras with unital *-homomorphisms.

\begin{proposition}{\cite[Lemma 8.1.4]{Lor97}}\label{proj}
Both $\A_0$ and $\A$ are projective. 
\end{proposition}
The projectivity of $\A$ follows from that of $\A_0$, which is a consequence of the fact that any contraction in a $C^*$-quotient lifts to a contraction.

\begin{remark}\label{C(F2)notproj} We mentioned in the introduction that $C^*(\mathbb{F}_2)$ is not projective. This means exactly that unitaries do not in general lift to unitaries, e.g., unitaries in the Calkin algebra with nonzero Fredholm index. 
\end{remark}

\begin{remark}\label{nounit}
By \cite[Proposition 2.3]{Bla85}, a projective $C^*$-algebra is contractible, which implies that it is non-unital and has trivial shape in the sense of \cite{Bla85} and trivial K-theory. 
(We say a unital $C^*$-algebra is contractible when it is homotopy equivalent to $\C$.) 
That $\A$ is contractible and has trivial $K_1$ is also proved in \cite[Theorem 3.1]{BN12}.
\end{remark}

Projectivity is a powerful (and rare) property because many other nice properties can be rephrased in terms of lifting problems. For instance, we can quickly establish Propositions \ref{P:hadwin} and \ref{P:compact}. 
Proposition \ref{P:hadwin} is a special case of the following. 
\begin{proposition}\label{projRFD}
Any separable projective $C^*$-algebra is RFD. 
\end{proposition}
See \cite[Theorem 11.2.1]{Lor97} or \cite{Had14} for a full proof.  The main ideas are that any separable $C^*$-algebra is the quotient of an RFD $C^*$-algebra, and that residual finite dimensionality passes to subalgebras.  For projective objects, being a quotient implies being a subobject.

Meanwhile Proposition \ref{P:compact} follows from the next proposition, the proof of which uses only the statement of \cite[Proposition 1]{Cho80} as opposed to its proof. 
\begin{proposition}\label{essential}
 Any separable projective $C^*$-algebra $B$ has no nontrivial projections, and if $(\pi, \H)$ is a faithful representation of $B$, then $\pi$ is essential (i.e., $\pi(B)$ contains no nonzero compact operators).
\end{proposition}

\begin{proof}
It will suffice to prove the unital case. To that end, let $B$ be a separable unital $C^*$-algebra that is projective as a unital $C^*$-algebra. 
Then $B$ is isomorphic to a quotient of $C^*(\F_\infty)$. By projectivity, this isomorphism lifts to an embedding of $B$ into $C^*(\F_\infty)$. Since $C^*(\F_\infty)$ has no nontrivial projections \cite[Theorem 1]{Cho80}, neither does $B$.\footnote{This fact also follows from \cite[Proposition 3]{HS17}.} It then follows, by the same argument as for \cite[Corollary 2]{Cho80}, that any faithful representation of $B$ on a Hilbert space $\H$ trivially intersects $K(\mathcal{H})$. 
\end{proof}

Contractibility makes it difficult to distinguish projective $C^*$-algebras, such as $\A$ and the free product $C[-1,1]\ast_\C C[-1,1]\simeq C^*_u\langle x_1,x_2: \|x_i\|\leq 1, x_i=x_i^*, i=1,2\rangle$. (That this particular $C^*$-algebra is projective follows from \cite[Theorem 3.2]{EL}.)  

\begin{question}\label{Q:d-s}
Are $\A$ and $C[-1,1]\ast_\C C[-1,1]$ isomorphic?  
\end{question}

We point out why the most obvious guess at an isomorphism does not answer Question \ref{Q:d-s}.  Let $x$ be the distinguished generator of $\A$, and $b,c$ be the canonical generators of the $C[-1,1]$ factors in the free product.  Since $\text{Re }x$ and $\text{Im }x$ are self-adjoint contractions with spectrum $[-1,1]$, one might propose an isomorphism via $\text{Re }x \mapsto b$, $\text{Im }x \mapsto c$.  But this fails, because $x = \text{Re }x + i \text{Im }x$ is a contraction, while $b + ic$ has norm 2.  (One can deduce this from the representation of $C[-1,1]\ast_\C C[-1,1]$ sending $b$ to $\begin{pmatrix} 0 & 1\\ 1 & 0\end{pmatrix}$, $c$ to  $\begin{pmatrix} 0 & i\\ -i & 0\end{pmatrix}$, and $b+ic$ to $\begin{pmatrix} 0 & 0\\ 2 & 0\end{pmatrix}$.)

Consider the ``abelianized" version of this question.  The universal unital $C^*$-algebra of a normal contraction is $C(\overline{\mathbb{D}})$, and the tensor product $C[-1,1] \otimes C[-1,1]$ is $C([-1,1]^2)$.  These two $C^*$-algebras are isomorphic, because the closed unit disk is homeomorphic to the closed (solid) $2\times2$ square.  But the homeomorphism is more complicated than simply rescaling coordinates.  Question \ref{Q:d-s} is asking for a noncommutative version of this homeomorphism.  With poetic license: are the noncommutative disk and noncommutative solid square homeomorphic?

\section{Finite-dimensional and nilpotent representations of $\A$}\label{nilp}

In this section we discuss the rather surprising fact that a universal contraction can be a countable direct sum of nilpotent matrices.  This was originally proved by Herrero via intricate computations; here we give two alternative arguments.  The first, ``\`{a} la Choi" and fairly short, proceeds by dilating slightly rescaled finite-dimensional compressions of a universal contraction to contractive nilpotent matrices.  The second, which we later generalize in the proof of Theorem \ref{nilrow}, asymptotically factorizes a faithful representation of $\A$ through the universal $C^*$-algebras of nilpotents of increasing order, relying on the classical nilpotent approximation theorem of Apostol-Foias-Voiculescu and a result of Shulman on lifting nilpotent contractions. 

\begin{theorem}\label{nilprep}
There exists a separating family of finite-dimensional representations of $\A$ that map the generator to contractive nilpotent matrices. In other words, for any noncommutative *-polynomial $q$ and contractive Hilbert space operator $S$, 
\begin{align*}
    \|q(S)\|&\leq 
    \sup \{\|q(M)\|: M\in \mathbb{M}_n, \|M\|\leq 1, \, M^n=0, \, n\geq 1\}.\\
\end{align*}
\end{theorem}

Theorem \ref{nilprep} gives us a different refinement (from Theorem \ref{FNT}) of the von Neumann-type inequality we got from the residual finite-dimensionality of $\A$ in Section \ref{max}.


Since the representations guaranteed by Theorem \ref{nilprep} are separating, by \cite[Theorem 3.2]{CS}, $\A$ has a dense subset of elements that attain their norm under one of these representations. However, this dense subset need not contain the canonical dense subset $\mathcal{P}_*$ of noncommutative *-polynomials in the generator, as the next example shows. 

\begin{example}\label{notnilp}
Let $x$ be the canonical generator of $\A$. Then $x+x^*$ cannot attain its norm under a finite-dimensional representation mapping $x$ to a nilpotent operator. 

First note that $\|x+x^*\|=2$, which can be seen by mapping $x$ to $1$. Now, suppose $M\in \mathbb{M}_n$ is nilpotent with $\|M\|=1$. Since $M+M^*$ is self-adjoint, if $\|M+M^*\|=2$, then there exists a unit vector $v$ such that either 
$$(M+M^*)v=2v \ \text{or}\ (M+M^*)v=-2v.$$
Because the unit vector $v$ is an extreme point of the unit ball, it follows that $Mv = v$ or $Mv = -v$, either way violating the fact that $\sigma(M) = 0$.
\end{example}

Now, let us commence with the proofs of Theorem \ref{nilprep}. The first proof is truly due to Herrero, though he does not draw this explicit conclusion. 
Let $T$ be a universal contraction as constructed in Proposition \ref{P:hadwin}(2). Herrero shows in \cite[Corollary 4.8]{Her81} that $T$ is the norm limit of contractive nilpotent operators that are block diagonal. Hence by Proposition \ref{P:aue}, the same holds for any universal contraction on $B(\ell^2)$, and we have the following extension of \cite[Corollary 4.8]{Her81}.
\begin{proposition}\label{BDN}
Every universal contraction in $B(\ell^2)$ is the norm limit of block-diagonal nilpotent contractive operators with finite-dimensional blocks. 
\end{proposition}
Theorem \ref{nilprep} follows by taking the family of representations induced by mapping the generator $x$ to each of the blocks in the sequence. 

\begin{remark}\label{inside} 
Note that no universal contraction $T\in B(\H)$ can be the norm limit of nilpotent contractions that lie \underline{in} $C^*(T,I)$. Indeed, if $N\in C^*(T,I)$ is nilpotent and $\pi:C^*(T,I)\to \C$ maps $T\mapsto 1$, then $$\|T-N\|\geq \|\pi(T)-\pi(N)\|=\|1-0\|=1.$$ 
Thanks to Don Hadwin for originally pointing out this fact.
\end{remark}

Our second argument proves Theorem \ref{nilprep} directly by an ``asymptotic dilation" argument.  This is close in spirit to Choi's proof \cite[Theorem 7]{Cho80} that $C^*(\F_2)$ is RFD, but the details are somewhat different.\footnote{Arveson has characterized in \cite[Theorem 1.3.1]{Arv72} the contractive Hilbert space operators that can be power dilated to contractive nilpotent operators -- for powers less than the order of the nilpotent.  But we need all *-monomials, at least asymptotically.}

\begin{proof}
For each $n>1$, let $J_n\in \mathbb{M}_n$ denote the $n\times n$ Jordan block with eigenvalue $0$. 
By \cite[Proposition 1]{HD}, there is a unit vector $\xi_n\in \C^n$ so that $w(J_n)=\langle J_n \xi_n, \xi_n\rangle=\cos\frac{\pi}{n+1}$, where $w(J_n)$ is the numerical radius of $J_n$. Let $U_n\in \mathbb{M}_n$ be a unitary so that $U_ne_1=\xi_n$ (where $e_1$ is from the standard basis of $\C^n$). For each $n > 1$, define $$M_n:=U_n^*J_nU_n=(a_{ij}^{(n)})_{1\leq i,j\leq n}.$$
Then $M_n$ is a nilpotent contraction with $a^{(n)}_{11}=\cos \frac{\pi}{n+1}$. 

Let $\H=\ell^2$ with the standard basis $\{e_1,e_2,...\}$; let $P_n$ be the orthogonal projection of $\H$ onto span$\{e_1,...,e_n\}$; let $T$ be a universal contraction operator on $\H$; and let $T_n:=P_nTP_n$ for each $n>1$. 
Then $T_n\otimes M_n\in P_nB(\H)P_n\otimes \mathbb{M}_n$ is a nilpotent contraction, and, if we view $P_n B(\H) P_n$ as $\mathbb{M}_n$, 
then the assignments $T\mapsto T_n\otimes M_n$ induce a family of finite-dimensional representations $\pi_n: C^*(T, I)\to \mathbb{M}_{n^2}$ where $\pi_n(T)$ is nilpotent. 
To see that this family is separating, it will suffice to show that $\oplus_n \pi_n$ is isometric, which will follow from showing that $\|q(T)\|=\sup_n \|q(T_n\otimes M_n)\|$ for any noncommutative *-polynomial $q$. To that end, let $q$ be any nonzero noncommutative *-polynomial, and let $\varepsilon>0$. 

For each $n> 1$, define $C_n=(c^{(n)}_{ij})\in \mathbb{M}_{n-1}$ by $c^{(n)}_{ij}:=a^{(n)}_{i+1,j+1}$ for each $1\leq i,j\leq n-1$. 
Since $a^{(n)}_{11}=\cos \frac{\pi}{n+1}\to 1$ as $n\to \infty$, we have 
$$\|[a_{1j}^{(n)*}]_{2\leq j\leq n}\|_{\ell^2_n}\to 0\ \text{and}\  \|[a_{i1}^{(n)}]_{2\leq i\leq n}\|_{\ell^2_n}\to 0$$
as $n\to \infty$, and so 
\begin{align*}
\|M_n-\left(\cos(\tfrac{\pi}{n+1})\oplus C_n \right)\|&=
\left\|\begin{pmatrix} 0 & a^{(n)}_{12} & ... & a^{(n)}_{1n}\\
           a^{(n)}_{21} & 0 & ... & 0\\
           \vdots & \vdots &... &\vdots \\
           a^{(n)}_{1n} & 0 &... & 0
    \end{pmatrix}\right\|\\
&\leq \|[a_{1j}^{(n)*}]_{2\leq j\leq n}\|_{\ell^2_n}+ \|[a^{(n)}_{i1}]_{2\leq i\leq n}\|_{\ell^2_n}\to 0
\end{align*}
as $n\to \infty$, which implies that 
$$\|T_n\otimes M_n- T_n\otimes (\cos(\tfrac{\pi}{n+1})\oplus C_n)\|\to 0$$
as $n\to \infty$.
Now, $q$ is Lipschitz on the unit ball of $\mathbb{M}_{n^2}\simeq P_nB(\mathcal{H})P_n\otimes \mathbb{M}_n$ with Lipschitz constant independent of $n$. From this we see that

\begin{align}\label{a}
\|q(T_n\otimes M_n)-q((T_n\otimes (\cos(\tfrac{\pi}{n+1})\oplus C_n))\|\xrightarrow[n\to\infty]{}0.    
\end{align}
Assume for simplicity that $\|q(T)\|=1$. Since $T_n\xrightarrow{S^*OT}T$, so does $\cos(\frac{\pi}{n+1})T_n$, and we have 
\begin{align}\label{b}
    \|q(\cos(\tfrac{\pi}{n+1})T_n)\|> 1-\frac{\varepsilon}{2}
\end{align} for $n$ sufficiently large.  
Combining \eqref{a} and \eqref{b}, we have, for $n$ sufficiently large, 
\begin{align*}
    \|q(T_n\otimes M_n)\|&> \|q(T_n\otimes (\cos(\tfrac{\pi}{n+1})\oplus C_n))\|-\frac{\varepsilon}{2}\\
    &= \|q((T_n\otimes \cos(\tfrac{\pi}{n+1}))\oplus (T_n\otimes C_n))\|-\frac{\varepsilon}{2}\\
    &= \|q((T_n\otimes \cos(\tfrac{\pi}{n+1}))\oplus q(T_n\otimes C_n)\|-\frac{\varepsilon}{2}\\
    &\geq  \|q(\cos(\tfrac{\pi}{n+1})T_n)\|-\frac{\varepsilon}{2}\\
    &>1-\frac{\varepsilon}{2}-\frac{\varepsilon}{2}=1-\varepsilon.
\end{align*}
Since $\varepsilon$ was arbitrary, this shows that $\oplus_n \pi_n$ is an isomorphism.
\end{proof}

Our final proof of Theorem \ref{nilprep} will require a few preliminary results. Though it is a little more work on the outset, this approach is what we will generalize to prove Theorem \ref{nilrow}, at which point we will appreciate the preliminaries established here. We begin by recalling a few definitions.

\begin{definition}
A separable $C^*$-algebra $A$ is {\it quasidiagonal} (QD) if there exists a sequence of completely positive contractive maps $\phi_n:A\to \mathbb{M}_{k_n}$ that are asymptotically multiplicative and asymptotically isometric. 
\end{definition}
It is easy to see that a separable RFD $C^*$-algebra is QD since it embeds faithfully into some product $\prod_{n} \mathbb{M}_{k_n}$. 

\begin{definition}
A set of operators $\Omega\subseteq B(\H)$ is called {\it quasidiagonal} (QD) if 
there exists an increasing sequence of finite-rank projections $P_1\leq P_2\leq...$ converging strongly to $I$ such that $\|P_nT-TP_n\|\to 0$ for every $T\in \Omega$.

A representation $\pi$ of a $C^*$-algebra $A$ on $\H$ is called {\it quasidiagonal} (QD) if $\pi(A)$ is a QD set of operators on $\H$. 
\end{definition}

Note that $\pi$ being a QD representation of $A$ is not equivalent to $\pi(A)$ being a QD $C^*$-algebra (see \cite[Remark 7.5.3]{BO}). The issue arises from a Fredholm index obstruction, which can be avoided by taking essential representations, i.e., those in which the image contains no nonzero compact operators. In fact, Voiculescu has shown in \cite{Voi91} that a separable $C^*$-algebra $A$ is QD if and only if every faithful essential representation of $A$ (on a separable Hilbert space) is quasidiagonal. 

\begin{definition}
An operator $T$ on $\H$ is called {\it quasitriangular} if there exists an increasing sequence of finite-rank projections $P_1\leq P_2\leq...$ converging strongly to $I$ such that $\|TP_n-P_nTP_n\|\to 0$; it is called {\it bi-quasitriangular} if $T$ and $T^*$ are quasitriangular.
\end{definition}

A QD operator is automatically bi-quasitriangular, and so if $\pi$ is a QD representation of a $C^*$-algebra $A$, then every $\pi(a)\in \pi(A)$ is bi-quasitriangular. In particular, if $\pi$ is an essential representation of a separable RFD $C^*$-algebra, then every $\pi(a)\in \pi(A)$ is bi-quasitriangular. 

Our first step is an observation which utilizes 
Apostol, Foias, and Voiculescu's characterization of the norm closure of nilpotents in $B(\H)$ \cite[Theorem 2.7]{AFV} (which is also a key ingredient in the proof of Herrero's result \cite[Corollary 4.8]{Her81}).  

\begin{proposition}\label{nillim}
  If $\pi$ is a faithful representation of a projective $C^*$-algebra $A$ on a separable Hilbert space $\H$, then for any $a\in A$, $\pi(a)$ is the norm limit of nilpotent operators in $B(\H)$ if and only if its spectrum is connected and contains $0$.
  
  In particular, if $T\in B(\H)$ is a universal contraction, then there exists a sequence $(T_n)_{n\geq 1}$ of nilpotent operators in $B(\H)$ that converge to $T$ in norm. Moreover, these can be chosen to be contractions satisfying $T^n_n=0$ for each $n\geq 1$. 
\end{proposition}

\begin{proof}
It will suffice to prove the unital case.
Let $\pi$ be a faithful representation of $A$, which is also essential by Proposition \ref{essential}. By \cite[Theorem 2.7]{AFV}, we know that an operator $T\in B(\mathcal{H})$ is the norm limit of nilpotent operators if and only if $T$ is bi-quasitriangular, $\sigma(T)$ and $\sigma_e(T)$ are connected, and $0\in \sigma_e(T)$. 
Because $A$ is projective, it is RFD. 
Since $\pi$ is essential, every element $\pi(a)\in \pi(A)$ is automatically bi-quasitriangular, and $\sigma(\pi(a))=\sigma_e(\pi(a))$.
Hence we have that $\pi(a)$ is the norm limit of nilpotent operators in $B(\mathcal{H})$ if and only if $\sigma(\pi(a))$ is connected and contains zero. 

In the case of $\A$, note that the spectrum of the generator must be $\overline{\D}$, and any sequence of nilpotent operators converging in norm to $T$ can be easily adjusted to satisfy the extra criteria. 
\end{proof}

Of course, Proposition \ref{BDN} says we can actually arrange it so that each $T_n$ is also block-diagonal. However, this will not be of any help when we generalize Theorem \ref{nilprep} in Section \ref{row}, and Proposition \ref{nillim} is really all we need. 

The following theorem gives an ``exactness"-type result, which says that any faithful representation of $\A$ on a separable Hilbert space ``asymptotically factorizes" through the family of universal $C^*$-algebras generated by nilpotent contractions of increasing finite orders.  

\begin{theorem}\label{approx}
Let $\pi:\A\to B(\H)$ be a faithful 
representation of $\A$ on a separable Hilbert space $\H$, and let $$\phi_n:\A \to \A_n:=C^*_u\langle x_n: \|x_n\|\leq 1,\ x_n^n=0\rangle$$ be the unital *-homomorphisms induced by mapping $x\mapsto x_n$ for each $n\geq 1$. Then there exists a sequence of *-homomorphisms $$\psi_n:\A_n\to B(\H)$$ such that $\psi_n\circ\phi_n$ converges to  $\pi$ pointwise in norm.
\end{theorem}

\begin{proof}
By Proposition \ref{nillim} there exists a sequence $(T_n)\subseteq B(\H)$ that converges in norm to $\pi(x)$ and satisfies $\|T_n\|\leq 1$ and $T_n^n=0$ for each $n\geq 1$.
For each $n\geq 1$, let $$\psi_n:\A_n\to C^*(T_n,I)\subseteq B(\H)$$ be the canonical surjective unital *-homomorphsim induced by mapping $x_n\mapsto T_n$. 
Then, for any $a\in \A$, 
$$\|\psi_n\circ \phi_n(a)-a\|\to 0.\qedhere$$
\end{proof}

\begin{remark}
By Remark \ref{inside}, the identity map on $\A$ does not factor through the family $\{A_n\}_{n\geq 1},$ which means we cannot hope to extend this to a ``nuclearity" result.
\end{remark}

With Theorem \ref{approx}, we can give our final proof of Theorem \ref{nilprep}. 

\begin{proof}
The family $\{\phi_n\}$ from Theorem \ref{approx} separates the elements of $\A$. By \cite[Theorem 5]{Shu08}, for each $n \geq 1$ the $C^*$-algebra $\A_n$ is projective and hence also RFD by Proposition \ref{projRFD}.  Let $\{\rho_{n,k}\}_{k\in \N}$ be a separating family of finite-dimensional representations of $\A_n$.  Then the compositions 
$\{\rho_{n,k}\circ\phi_n\}_{n,k\in\N}$ form a separating family of finite-dimensional representations of $\A$, and, moreover, $\rho_{n,k}\circ\phi_n(x)$ is nilpotent for each $n, k\in \N$ where $x$ denotes the generator of $\A$. 
\end{proof}

We conclude this section with a slight generalization of Theorem \ref{nilprep}.

\begin{theorem}\label{lambda}
For each $\lambda\in \D$, $\A$ has a separating family of finite-dimensional representations that send the generator to contractive matrices with spectrum $\{\lambda\}$.
\end{theorem}

The proof due to Herrero can be adapted to show this by replacing \cite[Corollary 4.8]{Her81} with the remark in \cite{Her81} just after it. 
In the remainder of this section we sketch how our second proof can also be adapted to prove Theorem \ref{lambda}.  

As it turns out, by \cite[Theorem 10]{LS14}, the universal $C^*$-algebra
$$\A_{\lambda,n}=C^*_u\langle x_n: \|x_n\|\leq 1,\ (x_n-\lambda)^n=0\rangle$$
is RFD for each $n\geq 1$.
So all we need is another ``exactness" result, like Theorem \ref{approx}, which says that every faithful representation $(\pi, \H)$ of $\A$ asymptotically factorizes through the $\A_{\lambda,n}$.
To define the maps $\psi_n$ as in Theorem \ref{approx}, it was crucial that a universal contraction $T$ is the norm limit of nilpotent contractions $T_n$, each with appropriate nilpotency order. For this more general setting, we need $T$ to be the norm limit of contractions $T_n$ satisfying $(T_n-\lambda I)^n=0$ for each $n\geq 1$, which we show in Lemma \ref{lem} below.  The rest of the argument for Theorem \ref{lambda} will then run like the third proof of Theorem \ref{nilprep}. 

\begin{lemma}\label{lem}
Let $\H$ be a separable Hilbert space, $T\in B(\H)$ be a universal contraction operator, and $\lambda\in \D$. Then there exists a sequence of contractive operators $(T_n)\subseteq B(\H)$ that converges in norm to $T$ and satisfies $(T_n-\lambda I)^n=0$ for each $n\geq 1$.
\end{lemma}

\begin{proof}
Since $\sigma(T-\lambda I)=\bar{\D}-\lambda$ is connected and contains $0$, by Proposition \ref{nillim}, there is a sequence $(N_n)\subseteq B(\H)$ of nilpotents so that $N_n\to T-\lambda I$ in norm and $N_n^n=0$ for each $n\geq 1$. Then $N_n+\lambda I\to T$ in norm; 
however, each $N_n+\lambda I$ may not be a contraction. To remedy this, we define $T_n:=c_nN_n+\lambda I$ for each $n\geq 1$ where
 \begin{displaymath}
   c_n = \left\{
     \begin{array}{lr}
       1 & : \|N_n+\lambda I\|\leq 1 \\
       \frac{1-|\lambda|}{\|N_n+\lambda I\|-|\lambda|} & : \|N_n+\lambda I\|>1
     \end{array}.
   \right.
\end{displaymath} 
Then $(T_n-\lambda I)^n=(c_nN_n)^n=0$ for each $n\geq 1$. Since $\|N_n+\lambda I\|\to \|T\|=1$, it follows that $c_n\to 1$ and hence that $T_n\to T$ in norm. Moreover, if $\|N_n+\lambda I\|\leq 1$, then $T_n=c_nN_n+\lambda I=N_n+\lambda I$ is a contraction. If $\|N_n+\lambda I\|>1$, then $1>c_n$, and we have that 
\begin{align*}
    \|T_n\|&=\|c_nN_n+\lambda I\|=\|c_n(N_n+\lambda I)+(1-c_n)\lambda I\|\leq c_n\|N_n+\lambda I\|+|\lambda||1-c_n|\\
    &=c_n\|N_n+\lambda I\|+|\lambda|-|\lambda|c_n=c_n(\|N_n+\lambda I\|-|\lambda|)+|\lambda|=(1-|\lambda|)+|\lambda|=1.
\end{align*}
Thus $(T_n)$ is the desired sequence.
\end{proof}

\begin{remark}
We cannot replace $\lambda\in \D$ from Theorem \ref{lambda} with $\lambda \in \C\backslash\D$ in either of the preceding results. Of course, if $|\lambda|>1$, then no contraction $Y \in B(\H)$ can satisfy $(Y-\lambda I)^n=0$.   
And if $Y \in B(\H)$ is any contraction with $(Y-\lambda I)^n=0$ for some $n\geq 1$ and some $\lambda\in \T$, then $Y=\lambda I$. Indeed, by \cite[Lemma 4]{LS14} we know that $Y$ is unitarily equivalent to an upper triangular array $(y_{ij})_{i,j\geq 1}$ with $y_{ii}=\lambda$ for $i\geq 1$. But, since $\|Y\|\leq 1$, it follows that $y_{ij}=0$ for $i\neq j$, and $Y=\lambda I$. In other words, for any $\lambda\in \T$ and $n\geq 1$, $$C^*_u\langle x: \|x\|\leq 1, (x-\lambda)^n=0\rangle\simeq \C$$ via the map $x\mapsto \lambda$. 

This argument shows that $\|N+\lambda I\|> 1$ for any $\lambda\in \T$ and any nonzero nilpotent operator $N$. (Just take $y= N+\lambda I$, and run the previous argument as a contradiction.) But this fails if we assume only that $N$ is quasinilpotent, i.e., $\sigma(N)=\{0\}$. For example, let $V$ be the Volterra integration operator, and $N=(I+V)^{-1}-I$. We know from \cite[Problem 190]{Hal82} that $\|N+I\|=\|(I+V)^{-1}\|=1$, and $\sigma(N)=\{0\}$. However, $N\neq 0$. 
\end{remark}

\section{$\A$ and Connes' Embedding Problem} \label{CEP}
As we mentioned in the introduction, $\A$ often behaves much like $C^*(\F_2)$, in particular when it comes to the role of $C^*(\F_2)$ in Kirchberg's characterizations of Connes' Embedding Problem. In this section, we elucidate the relation of $\A$ to Connes' Embedding Problem and consider some related questions.
Before we begin, we must first establish some relevant background. 
\subsection{The WEP, LLP, and some results of Kirchberg}\label{WEPLLP}

For two $C^*$-algebras $A$ and $B$, $A\otimes_{max} B$ and $A\otimes_{min} B$ refer to the completion of the algebraic tensor product $A\odot B$ of $A$ and $B$ with respect to the maximal and minimal $C^*$-norms, respectively. The embedding $A\odot B\to A\otimes_{min} B$ extends uniquely to a surjective *-homomorphism $A\otimes_{max}B\to A\otimes_{min}B$. 
See \cite[Chapter 3]{BO} for definitions and other relevant properties, and see \cite[Section 2]{Bla85} for a short discussion on universal $C^*$-algebras and tensor products. 

Though it is readily verifiable that, for any $C^*$-algebras $A,B$ and $C$ with $A\subseteq B$, 
$$A\otimes_{min}C\subseteq B\otimes_{min} C,$$ this containment can easily fail for the maximal tensor product because subalgebras tend to have more representations.  In \cite{Lan73}, Lance characterized embeddings of $C^*$-algebras that are always preserved under the maximal tensor product in terms of completely positive extensions (see \cite[Proposition 3.6.6]{BO} for a proof). We say a linear map is {\it positive} if it sends positive elements to positive elements and {\it completely positive} if this remains true even after amplification with matrix algebras. If a completely positive map is contractive, we call it cpc, and if it is moreover unital, we call it ucp. 
We give Lance's result as a proposition/definition hybrid. 

\begin{proposition}\label{rwi}
Let $A$ and $B$ be unital $C^*$-algebras and $1_B\in A\subset B$ an inclusion. The inclusion is called \underline{relatively weakly injective} if the following equivalent conditions hold. 
\begin{enumerate}
\item There exists a ucp map $\phi: B\to \pi_u(A)''$ so that $\phi(a)=\pi_u(a)$ for all $a\in A$, where $(\pi_u,\mathcal{H}_u)$ is the universal representation of $A$ from the Gelfand-Naimark Theorem. 
\item For any $C^*$-algebra $C$, there is a natural inclusion $$A\otimes_{max} C\subseteq B\otimes_{max} C.$$
\end{enumerate}
\end{proposition}

The map $\phi:B\to \pi_u(A)''$ is called a {\it weak expectation}. 
If a $C^*$-algebra $A$ embeds relatively weakly injectively into $B(\mathcal{H})$ for any Hilbert space $\mathcal{H}$, then it is said to have Lance's {\it Weak Expectation Property} (WEP). 

\begin{proposition}\label{projrwi}
If $A$ is a projective $C^*$-algebra and $B$ surjects onto $A$, then $A$ embeds relatively weakly injectively into $B$.
\end{proposition}
Indeed, the identity map $id_A:A\to A$ lifts to a *-homomorphism $\iota:A\to B$, which must be injective since $id_A$ is.  The weak expectation $B \to \pi_u(A)$ is just the induced surjective *-homomorphism $B\to \pi_u(A)$. 

Let $A$ and $B$ be $C^*$-algebras and $J$ a closed two-sided ideal in $B$ with quotient map $\pi:B\to B/J$. A cpc map $\phi:A\to B/J$ is \textit{liftable} if there is a cpc map $\tilde{\phi}:A\to B$ such that $\pi\circ \tilde{\phi}=\phi$. We say a cpc map $\phi:A\to B/J$ is \textit{locally liftable} if for any finite dimensional operator system $S\subseteq A$ there is a cpc map $\tilde{\phi}:S\to B$ such that $\pi\circ \tilde{\phi}=\phi|_S$. 

A unital $C^*$ algebra $A$ has the \textit{(local) lifting property} or (L)LP if any ucp\footnote{Yes, we did unceremoniously shift from cpc maps to ucp maps, but it turns out to be sufficient to restrict ourselves to ucp maps. Using the fact that the set of liftable cpc maps from a separable operator system into a $C^*$-quotient is closed (\cite[Theorem 6]{Arv77}), one can argue (as in \cite[Lemma 13.1.2]{BO}) that any ucp map with a (local) cpc lift has a (local) ucp lift.} map from $A$ into a quotient $C^*$-algebra is locally liftable. A non-unital $C^*$-algebra has the (L)LP if and only if its unitization does.

We can view the LP as projectivity in the category of unital $C^*$-algebras with (unital) cpc maps. Then the LLP can be thought of as a ``localized ucp projectivity" and the WEP as ``weak injectivity." 

\begin{proposition}\label{P:rwi} \cite[Corollaries 2.6 and 3.3]{Kir93} The LLP and WEP pass to relatively weakly injective subalgebras.
\end{proposition}

A key class of examples comes from Kirchberg \cite[Lemma 3.3]{Kir94}. 

\begin{lemma}
For any free group $\F$, $C^*(\F)$ has the LLP. Moreover, $C^*(\F)$ has the LP if $\F$ has countably many generators. 
\end{lemma}






With this, we can readily prove the following proposition, which greatly simplifies checking whether or not a given $C^*$-algebra has the LLP. 
This is a mild rephrasing of \cite[Corollary 13.1.4]{BO}, and we include a brief argument for the reader's convenience. 

\begin{proposition}
A unital $C^*$-algebra $A$ has the LLP if and only if the identity on $A$ is locally liftable.
\end{proposition}

To see why it suffices to lift the identity of $A$, choose a free group $\F$ such that we have a surjection $\pi:C^*(\F)\to A$. 
Let $B$ be a $C^*$-algebra with two-sided closed ideal $J$ and $\phi:A\to B/J$ a ucp map. Suppose $S\subseteq A$ is a finite dimensional operator system, and let $\rho:S\to C^*(\F)$ be the lift of $id_A|_S$. Then $\phi\circ \pi:C^*(\F)\to B/J$ is a ucp map, which lifts to a ucp map $\psi:C^*(\F)\to B$. Our desired lift of $\phi|_S$ is $\psi\circ\rho$.

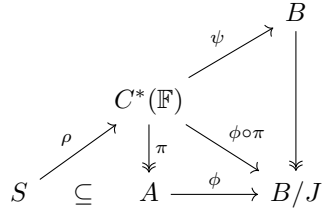
\[
\begin{tikzcd}
& & B\arrow[two heads]{dd}\\
& C^*(\F) \arrow[two heads]{d}{\pi} \arrow{rd}{\phi\circ\pi}\arrow{ru}{\psi} &\\
S\arrow{ur}{\rho} \arrow[r,phantom, "\rotatebox{360}{$\subseteq$}", description] & A \arrow{r}{\phi} & B/J\\ 
\end{tikzcd}
\]

Given the above proposition, the following is surely known to the experts. (The special case for $\A$ follows immediately from \cite[Proposition 2.2]{Oza01}.)

\begin{proposition}\label{projLLP'}
Any projective $C^*$-algebra has the LLP. In particular, $\A$ has the LLP. 
\end{proposition}




Next we give Kirchberg's deep and elegant characterization of Connes' Embedding Problem. 
Below is an augmented and abridged version of \cite[Proposition 8.1]{Kir93} (omitting some of the equivalent conditions and adding one that quickly follows). 
\begin{theorem}[Kirchberg]\label{Kir}
The following conjectures are equivalent. 
\begin{enumerate}
    \item Every finite von Neumann algebra with separable predual is embeddable into an ultrapower of the hyperfinite $II_1$-factor. 
    \item $C^*(\F)\otimes_{max}C^*(\F)\simeq C^*(\F)\otimes_{min}C^*(\F)$ canonically (where $\F$ is any nonabelian free group).
    \item $C^*(\F)\otimes_{max}C^*(\F)$ is RFD.
    \item The LLP implies the WEP.
    \item $C^*(\F)$ has the WEP. 
\end{enumerate}
\end{theorem}

Item (1) is traditionally referred to as Connes' Embedding Problem (CEP). 
Item (3) is not included in \cite{Kir93}; however, it is equivalent to (2) by virtue of the following proposition, which has often been alluded to in the literature. We briefly sketch an argument here; for a complete proof, see \cite[Proposition A.0.1]{Cou18}. 

\begin{proposition}\label{factors}
Given two RFD $C^*$-algebras $A$ and $B$, $A\otimes_{max} B$ is RFD if and only if $A\otimes_{max} B=A\otimes_{min} B$ canonically.
\end{proposition}

By taking tensor products of finite dimensional representations, one can readily verify that the minimal tensor product of two RFD $C^*$-algebras is again RFD. 
For the other implication, it will suffice to show that any finite dimensional representation of the maximal tensor product factors through the minimal tensor product. 

 To that end, let $A$ and $B$ be $C^*$-algebras and $\pi:A\otimes_{max} B\to \mathbb{M}_n$ be a finite dimensional representation with restrictions $\pi_A:A\to \mathbb{M}_n$ and $\pi_n:B\to \mathbb{M}_n$, i.e., $\pi_A$ and $\pi_B$ have commuting ranges and $\pi|_{A\odot B}=\pi_A\times \pi_B$. Then we have the induced *-homomorphism $\pi_A\otimes\pi_B:A\otimes_{min}B\to \pi_A(A)\otimes_{min}\pi_B(B)\simeq \pi_A(A)\otimes_{max}\pi_B(B)$. On the other hand, since $\pi_A(A)$ and $\pi_B(B)$ commute, the natural embeddings of $\pi_A(A)$ and $\pi_B(B)$ into $\mathbb{M}_n$ induce a representation of $\pi_A(A)\otimes_{max}\pi_B(B)$ such that $\pi_A(a)\otimes\pi_B(b)\mapsto \pi_A(a)\pi_B(b)=\pi(a\otimes b)$ for each $a\in A$ and $b\in B$. The composition of this representation with $\pi_A\otimes\pi_B$ gives a representation of $A\otimes_{min} B$ that agrees with $\pi$ on $A\odot B$. 



\subsection{$\A$ versus $C^*(\F_2)$} \label{A and F2}
Let us first show that $\A$ is ``smaller" than $C^*(\F_2)$.

\begin{lemma}\label{L:neq}
We have that $\A$ is a quotient of $C^*(\F_2)$, but not vice versa.
\end{lemma}

\begin{proof}
Let $x$ be the canonical generator of $\A$, with $x_1 + i x_2$ its decomposition into real and imaginary parts.  Then $\A = C^*(x,1) = C^*(x_1, x_2, 1) = C^*(e^{ix_1}, e^{ix_2})$, since the spectrum of $e^{ix_j}$ is contained in a proper arc of the circle, and a continuous branch of $-i \log (\cdot)$ takes $e^{ix_j}$ back to $x_j$.  Any $C^*$-algebra generated by two unitaries is a quotient of $C^*(\F_2)$.

To establish the other statement, it suffices to recall that $C^*(\F_2)$ is not singly generated (as a unital $C^*$-algebra).  For $C^*(\F_2)$ surjects onto $C(\T^2)$, since the latter is generated by two unitaries, and $C(\T^2)$ is not singly generated because $\T^2$ is not homeomorphic to a compact subset of the plane.
\end{proof}

Although $C^*(\F_2)$ and $\A$ are not isomorphic, they play the same universal role for CEP. This has been observed before, specifically by Pisier in \cite[Proposition 16.13]{Pis03}, where he shows that CEP has an affirmative answer if and only if $\A$ has the WEP (compare with Theorem \ref{Kir}(5)). It turns out that this can be linked to a strong structural relationship between the two algebras. 

\begin{theorem}\label{rwithm}
$\A$ and $C^*(\F_2)$ embed relatively weakly injectively into one another. 
\end{theorem}

\begin{proof}
We will use Proposition \ref{rwi}(1); in this case it will turn out that the required ucp maps can be taken to be *-homomorphisms.

From Lemma \ref{L:neq}, there exists a surjection $\pi: C^*(\F_2)\to \A$. So, $\A$ embeds relatively weakly injectively into $C^*(\F_2)$ by Proposition \ref{projrwi}. 

For the other relatively weakly injective embedding, let $\pi_u:C^*(\F_2)\to B(\mathcal{H}_u)$ be the universal representation of $C^*(\F_2)$, and let $U_1,U_2\in B(\mathcal{H}_u)$ be the images of the generators under $\pi_u$. 
Let $A_1,A_2\in B(\mathcal{H}_u)$ be self-adjoint elements such that $U_j=e^{i A_j}$ and 
\begin{center}
   $ C^*(U_1,U_2)\subseteq C^*(A_1, A_2)=C^*(\frac{A_1}{\alpha}+i\frac{A_2}{\alpha})\subseteq C^*(U_1,U_2)'',$
\end{center}
where $\alpha=\|A_1+iA_2\|$. Let $x = x_1 + ix_2$ be the canonical generator of $\A$ as before.  Then $e^{i\alpha x_j}$ are unitaries in $\A$ and 
\begin{align*}
    C^*(e^{i\alpha x_1}, e^{i\alpha x_2})\subseteq C^*(x_1, x_2)=\A.
\end{align*}
By universality, there exist surjective unital *-homomorphisms, 
\begin{center}
    $\phi:C^*(U_1,U_2)\to C^*(e^{i\alpha x_1}, e^{i\alpha x_2})\ \text{and}\ \psi:\A \to C^*(\frac{A_1}{\alpha}+i\frac{A_2}{\alpha})$
\end{center}
with $\phi(U_j)=e^{i\alpha x_j}$ and $\psi(x_j)=\frac{A_j}{\alpha}$. 
Then 
\begin{align*}
    \psi\circ \phi(U_j)=\psi(e^{i\alpha x_j})=e^{i\alpha (\frac{1}{\alpha}A_j)}=e^{iA_j}=U_j.
\end{align*}
It follows that $\psi\circ\phi=id_{C^*(U_1,U_2)}$, and this completes the proof.
\end{proof}





We now leverage relative weak injectivity to get the following characterizations of CEP.  Items (1) and (3) were deduced differently by Pisier \cite[Proposition 16.13]{Pis03} and Blecher-Duncan \cite[Section 6.1]{BD}, respectively.


\begin{theorem}\label{QWEPlem}
The following are equivalent to CEP. 
\begin{enumerate}
\item $\A$ has the WEP.
\item Projectivity implies the WEP.
\item $\A\otimes_{max}\A\simeq \A\otimes_{min}\A$ canonically.
\item $\A\otimes_{max} \A$ is RFD.
\item $\A\otimes_{max}\A$ has a separating family of finite-dimensional representations that map the two generators $x\otimes 1$ and $1\otimes x$ to nilpotent matrices. 
\end{enumerate}
\end{theorem}

\begin{proof}
(2) $\Rightarrow$ (1) $\iff$ CEP $\Rightarrow$ (2): Condition (2) above implies (1) above by the fact that $\A$ is projective (Proposition \ref{proj}).  By Theorem \ref{rwithm} and Proposition \ref{P:rwi}, $\A$ has the WEP if and only if $C^*(\F_2)$ does, which is item (5) in Theorem \ref{Kir}.  Item (4) in Theorem \ref{Kir} implies (2) above because projective algebras have the LLP (Proposition \ref{projLLP'}).


(3) $\iff$ (4) $\iff$ CEP: 
The equivalence of (3) and (4) follows from Proposition \ref{factors}.
With Theorem \ref{rwithm} and a few applications of Lance's characterization of relatively weakly injective embeddings (Proposition \ref{rwi}), we have the embeddings
$$\A\otimes_{max} \A\subset C^*(\F_2)\otimes_{max} C^*(\F_2)$$
and 
$$C^*(\F_2)\otimes_{max} C^*(\F_2)\subset \A\otimes_{max} \A.$$
Since residual finite dimensionality passes to subalgebras, we have that (4) is equivalent to $C^*(\F_2)\otimes_{max} C^*(\F_2)$ being RFD (item (3) of Theorem \ref{Kir}).   

(4) $\iff$ (5): The reverse direction is obvious, so we assume (4) and show (5).  Since (3) and (4) have already been shown to be equivalent, it suffices to prove the claim in (5) for $\A\otimes_{min}\A$, and this essentially amounts to the argument that the minimal tensor product of two RFD $C^*$-algebras is RFD.  By Theorem \ref{nilprep}, for $i=1,2$, there exist separating families of finite-dimensional representations $\{\sigma^{(i)}_n\}_{n\in \N}$ of $\A$ where each $\sigma^{(i)}_n$ maps the generator to a nilpotent matrix. 
For $n,m\geq 1$, let $\pi_{n,m}=\sigma_n^{(1)}\otimes \sigma_m^{(2)}$ denote the *-homomorphism on $\A\otimes_{min} \A$ that maps $a\otimes b$ to $\sigma_n^{(1)}(a)\otimes \sigma_m^{(2)}(b)$. 
Then $\{\pi_{n,m}\}_{m,n\in \N}$ form the desired family of representations.
\end{proof}

\begin{remark}
To highlight the difficulty in proving part (3) of Theorem \ref{QWEPlem}, consider the following. 
Because residual finite dimensionality is preserved by extensions, $\A\otimes_{max}\A$ is RFD if and only if $\A_0\otimes_{max}\A_0$ is. Moreover, by Proposition \ref{factors} $\A_0\otimes_{max}\A_0$ is RFD if and only if it is canonically isomorphic to $\A_0\otimes_{min}\A_0$. However, recall that $\A_0$ is contractible, and contractibility is preserved under taking tensor products (just take the homotopy between the zero map and the identity on the contractible $C^*$-algebra and tensor it with the identity on the other $C^*$-algebra). In particular, this means that both $\A_0\otimes_{min}\A_0$ and $\A_0\otimes_{max}\A_0$ are contractible, which makes them very difficult to distinguish with topological invariants. 
The authors would like to thank Hannes Thiel for clarifying some points on this. 
\end{remark}

\begin{remark}
Using their universality, one can identify $\A\otimes_{max}\A$ with the universal unital $C^*$-algebra generated by a pair of doubly commuting contractions (i.e., the generators commute and commute with each other's adjoints). 
In \cite{CR} it was asked whether the universal $C^*$-algebra generated by a pair of commuting contractions, which can be identified with the maximal $C^*$-algebra for the bidisk algebra \cite[Example 2.3]{BP91}, is RFD. 
An affirmative answer would amount to a noncommutative *-polynomial analogue (in the spirit of \eqref{E:*vN}) to And\^{o}'s von Neumann Inequality from \cite{And63}.
It does not seem to be known whether this question is easier than the question for $\A\otimes_{max}\A$.

The universal (unital) $C^*$-algebra generated by a pair of unrelated contractions is a full (unital) free product of RFD $C^*$-algebras, which is RFD by \cite[Theorem 3.2]{EL}. 
\end{remark}

\begin{remark}\label{QD}
Though the question of whether or not $\A\otimes_{max}\A$ is RFD proves to be quite challenging, the proof of \cite[Proposition 7.4.5]{BO} can be readily adapted to show that $\A\otimes_{max}\A$ is quasidiagonal. 
\end{remark}



\section{$C^*$-algebras generated by multiple universal contractions}\label{remarks}


We conclude with a few remarks on $C^*$-algebras generated by two or more universal contractions. 

\subsection{Universal row contractions}\label{row}
We call an $m$-tuple of operators $(S_1,...,S_m)\in B(\H)^m$ an \textit{$m$-row contraction} if $\|\sum S_i S_i^*\|\leq 1$. (Note that the $S_i$ are not assumed to be commuting.) Such tuples have received considerable attention in multivariate operator theory, especially with regard to generalizing the commutant lifting theorem and von Neumann's inequality (e.g., \cite{Bun84}, \cite{Fra82}, \cite{Fra84}, \cite{Pop89}, and \cite{Pop91}).

 In this section, we generalize Theorems \ref{FNT} and \ref{nilprep} to universal $C^*$-algebras generated by 
 $m$-row contractions\footnote{Note that the row contraction relation automatically enforces a norm condition on the generators. As we mentioned in Section \ref{Projectivity}, this will guarantee that these universal $C^*$-algebras exist (and are separably acting).}: 
$$\R_m:=C^*_u\left\langle y_1,...,y_m: \left\|\sum_{i=1}^m y_iy_i^*\right\|\leq 1\right\rangle.$$ 
Some of the techniques we have developed for $\A=\R_1$ can be generalized to prove the same results for $\R_m$. 
In particular, where before we had von Neumann-type inequalities for noncommutative *-polynomials in one variable, in this section we obtain a noncommutative *-polynomial analogue to Popescu's von Neumann inequality for 
row contractions \cite[Theorem 2.1]{Pop91}.  



\begin{proposition}\label{universal}
Fix $1<m<\infty$ and let $\{y_1,...,y_m\}$ denote the generators of $\R_m$. Then for any $1\leq k<m$, any subset $\{y_{i_1},...,y_{i_k}\}\subseteq\{y_1,...,y_m\}$ forms a universal $k$-row contraction. In particular, $\|\sum_{j=1}^k y_{j_i}y_{j_i}^*\|=1$, and, moreover, each $y_i$ is a universal contraction. 
\end{proposition}

\begin{proof}
Since 
$$0\leq \sum_{j=1}^k y_{j_i}y_{j_i}^*\leq \sum_{j=1}^m y_jy_j^*,$$
it follows that the $\{y_{i_1},...,y_{i_k}\}$ form a $k$-row contraction. For any $k$-row contraction $(Z_1,...,Z_k)$, the restriction to $C^*(y_{i_1},...,y_{i_k})$ of the map $\R_m\to C^*(Z_1,...,Z_k,1)$ sending $y_{i_j}\to Z_j$ for $1\leq j\leq k$ and the rest of the generators to $0$ is a surjection. 
\end{proof}
By \cite[Theorem 5.1]{LS12}, $\R_m$ is projective for $1\leq m<\infty$, just like $\A$. From this it follows that $\R_m$ is RFD (see \cite[Example 3]{CR} for a more direct proof). It turns out we can say moreover that the noncommutative *-polynomials in the generators of $\R_m$ form a dense subset of elements that actually attain their norms under some finite-dimensional representation (just as with $\A$).

\begin{theorem}\label{FNT Rm}
Fix $1\leq m<\infty$, and let $q$ be a noncommutative *-polynomial in m variables of degree $d$. Then there exists a representation $\pi$ of $\R_m$ on a $(2m)^{d+1}$-dimensional Hilbert space such that $\|\pi(q(y_1,...,y_m))\|=\|q(y_1,...,y_m)\|$. 
\end{theorem}

\begin{proof}
To prove this, we will need to change little from the proof of Theorem \ref{FNT}. The key there was that compressions of $\pi_0(x)$ are still contractions, which means they induce representations of $\A$. Hence, the argument can be adapted for any universal $C^*$-algebra $C^*\langle \mathcal{G}| \mathcal{R}\rangle$ with a finite set of generators and with relations that are preserved under compressions. All we have to prove is that $\R_m$ is such an algebra. To that end, suppose $Y_1,...,Y_m\in B(\H)$ form a row contraction, and let $P\in B(\H)$ be a projection. Since $P\leq I_{\mathcal{H}}$, we have that $Y_iPY_i^*\leq Y_iY_i^*$ for each $1\leq i\leq m$. Then 
$$\left\|\sum_{i=1}^m PY_iPY_i^*P\right\|\leq \left\|\sum_{i=1}^m Y_iPY_i^*\right\|\leq \left\|\sum_{i=1}^m Y_iY_i^*\right\|\leq 1. \qedhere$$
\end{proof}



Now we turn our attention back to nilpotents.
In Section \ref{nilp}, we gave three arguments showing that the universal contraction can be faithfully represented as a countable direct sum of nilpotent matrices. 
The third can be generalized to show the same for $\R_m$, for $1\leq m<\infty$. The main idea of the proof is to establish an ``exactness"-type result like Theorem \ref{approx}, which says that any faithful representation of $\R_m$ ``asymptotically factorizes" though the family of universal $C^*$-algebras generated by order $n$ nilpotents that form row contractions:

\[\R_{m,n}=C^*_u\left\langle y_{1,n},...,y_{m,n}: \begin{tabular}{c}$\|y_{i,n}\|\leq 1, y^n_{i,n}=0, i=1,...,m$,\\ $\|\sum_{i=1}^m y_{i,n}y^*_{i,n}\|\leq 1$
\end{tabular}\right\rangle\] 

Just as in the proof of Theorem \ref{approx}, we will first need to know that any faithful representation of a universal 
row contraction is the norm limit of row contractions whose entries are each nilpotent. To that end, we have the following lemma., 

\begin{lemma}\label{rowlim}
Let $1\leq m<\infty$ and $\pi:\mathcal{R}_m\to B(\mathcal{H})$ a faithful representation on a separable Hilbert space. 
Then, for each $1\leq j\leq m$, there exists a sequence $(T_{j,n})_{n\geq 1}$ in $B(\H)$ that converges in norm to $\pi(y_j)$, and for each $n\geq 1$, $T_{j,n}^n=0$ and $(T_{1,n},...,T_{m,n})$ forms a row contraction. 
\end{lemma}

\begin{proof}
By Propositions \ref{universal} and \ref{nillim}, each $\pi(y_j)$ is the norm limit of a sequence of nilpotent contractions $(N_{j,n})_{n\geq 1}$, which we can assume satisfy $N_{j,n}^n=0$ for all $n\geq 1$. 
Now for each $n\geq 1$, we let $\alpha_n=\min\{1, \|\sum_{j=1}^m N_{j,n}N^*_{j,n}\|^{-1/2}\}$
(where we assume $\alpha_n=1$ if $N_{j,n}=0$ for all $1\leq j\leq m$), and define $$T_{j,n}:=\alpha_nN_{j,n}$$ for each $1\leq j\leq m$. We immediately have $T^n_{j,n}=\alpha^nN^n_{j,n}=0$ for all $n\geq 1$ and $1\leq j\leq m$. Moreover, since $N_{j,n}\to \pi(y_j)$ for each $1\leq j\leq m$, it follows from Proposition \ref{universal} that $\alpha_n\to 1$, and hence $T_{j,n}\to \pi(y_j)$ for each $1\leq j\leq m$. Finally we check that $(T_{1,n},...,T_{m,n})$ forms a row contraction for each $n\geq 1$. If $(N_{1,n},...,N_{m,n})$ is a row contraction, then $\alpha_n=1$. If $\|\sum_{j=1}^m N_{j,n}N^*_{j,n}\|>1$, then 
$$\left\|\sum_{j=1}^m T_{j,n}T^*_{j,n}\right\|=\left\|\sum_{j=1}^m \alpha^2_nN_{j,n}N^*_{j,n}\right\|=\left\|\sum_{j=1}^m N_{j,n}N^*_{j,n}\right\|^{-1}\left\|\sum_{j=1}^m N_{j,n}N^*_{j,n}\right\|=1.$$
\end{proof}

\begin{theorem}\label{nilrow}
For $1\leq m<\infty$, $\R_m$ has a separating family of finite-dimensional representations that map each generator to a nilpotent. In particular, there exist nilpotent matrices $N_{k_1},...,N_{k_m}\in \mathbb{M}_{j_k}$, $k\in \N$, such that $(\oplus_k N_{k_1},...,\oplus_k N_{k_m})$ is a universal 
$m$-row contraction. 
\end{theorem}

\begin{proof}
Let $\pi$ be a faithful representation of $\R_m$ on a separable Hilbert space $\H$, and let $(T_{i,n})_{n\geq 1}$ be the sequences of nilpotents guaranteed by Lemma \ref{rowlim} for each $1\leq i\leq m$.  

By universality, for each $n\geq 1$, we can define a surjective *-homomorphsim $\psi_n$ from 
\[\R_{m,n}=C^*_u\left\langle y_{1,n},...,y_{m,n}: \begin{tabular}{c}$\|y_{i,n}\|\leq 1, y^n_{i,n}=0, i=1,...,m$,\\ $\|\sum_{i=1}^m y_{i,n}y^*_{i,n}\|\leq 1$
\end{tabular}\right\rangle\] 
to $C^*(1,T_{1,n},...,T_{m,n})$ by sending $y_{i,n}\mapsto T_{i,n}$ for each $i=1,...,m$. Let $\phi_n:\R_m\to \R_{m,n}$ denote the surjection induced by sending generators to generators. Then, for each $i=1,...,m$, 
$$\|\psi_n\circ\phi_n(y_i)-\pi(y_i)\|\to 0,$$
which means that $\psi_n\circ\phi_n$ converge pointwise in norm to $\pi$. In particular, the $\{\phi_n\}$ separate the points of $\mathcal{R}_m$. By \cite[Theorem 2.3]{LS12'}, each $\R_{m,n}$ is projective and hence RFD by Proposition \ref{projRFD}. 
So, composing the $\{\phi_n\}$ with the separating families of finite-dimensional representations of the $\R_{m,n}$'s gives us the desired family of representations of $\R_m$.

The matrices in the statement of the theorem come from taking the direct sum of this family of representations. 
\end{proof}


\begin{remark}
Since the generators of $\R_m$ are all universal contractions, Example \ref{notnilp} also shows that Theorem \ref{FNT Rm} will not hold if we also want the generators to be mapped to nilpotents. 
\end{remark}

We conclude this section by rephrasing the results in terms of a Popescu-von Neumann inequality for noncommutative *-polynomials on 
row contractions. 

\begin{corollary}\label{P-vN}
Let $q$ be a noncommutative *-polynomial in $m$ variables of degree $d$ and $(T_1,...,T_m)$ a row contraction. Then 
\begin{align*}
    \|&q(T_1,...,T_m)\|\leq \sup\{\|q(A_1,...,A_m)\|: (A_1,...,A_m)\in (\mathbb{M}_n^m)_{\leq 1}, n\geq 1\}\\
    &=\max\{\|q(A_1,...,A_m)\|: (A_1,...,A_m)\in (\mathbb{M}_{(2m)^{d+1}}^{m})_{\leq 1}\}\\
    &=\sup\{\|q(N_1,...,N_m)\|: (N_1,...,N_m)\in (\mathbb{M}_n^m)_{\leq 1}, N_i^n=0, 1\leq i\leq m, n\geq 1\}.
\end{align*}
\end{corollary}

\subsection{Universal Pythagorean $C^*$-algebras}\label{pythagorean}

An $m$-tuple of operators $(S_1,...,S_m)$ in $B(\H)^m$ satisfying the identity 
$$\sum_{i=1}^m S_i^*S_i=1$$
is often called a column isometry. Such operators have recently been utilized in \cite{BJ} to construct several interesting unitary representations of Thompson's groups $F$ and $T$ (and their $n$-ary versions $F_n$ and $T_n$). We adopt the language of \cite{BJ} in calling 
$$\mathcal{P}_m=C^*_u\left\langle a_1,...,a_m: \sum_{i=1}^m a_i^*a_i=1\right\rangle,$$
universal Pythagorean $C^*$-algebras. 

Below, we adapt Choi's technique from \cite[Theorem 7]{Cho80} to show that each $\mathcal{P}_m$ is RFD for $m>1$. Note that this is not true for $m=1$.  The universal $C^*$-algebra of a single isometry is not RFD because every finite-dimensional isometry is unitary.  
\begin{theorem}\label{rowisomRFD}
 For each $m>1$, $\mathcal{P}_m$ is RFD.
\end{theorem}

\begin{proof}
First, we faithfully and nondegenerately represent $\mathcal{P}_m$ on a separable Hilbert space $\H$, denoting the generators by $A_1,...,A_m$. Let $(P_n)$ be an increasing sequence of projections, each of rank $n$, such that $P_n$ converges strongly to $I_{\mathcal{H}}$. 
For each $n\geq 1$ and $1\leq i\leq m$, 
let $B_i^{(n)}=P_nA_iP_n$ and $C_n= \sum_{i=1}^m B_i^{(n)*}B_i^{(n)}$, and note that $C_n\leq P_n$ for each $n\geq 1$. 

Define operators $A_1^{(n)},...,A_m^{(n)}\in B(P_n\H\oplus P_n\H)$ by 
$$A_1^{(n)}=\begin{pmatrix} B_1^{(n)} & 0\\ 0& P_n\end{pmatrix}, \hspace{.25 cm} A_2^{(n)}=\begin{pmatrix} B_2^{(n)} & 0\\ (P_n-C_n)^{1/2} & 0\end{pmatrix}, \hspace{.25 cm} \text{and} \hspace{.25 cm} A_i^{(n)}=\begin{pmatrix} B_i^{(n)} & 0 \\ 0 & 0\end{pmatrix}$$
for each $2<i\leq m$. For each $n$, we compute  
\begin{align*}
    \sum_{i=1}^m & A_i^{(n)*}A_i^{(n)}\\
    &=\begin{pmatrix} B_1^{(n)*}B_1^{(n)} & 0\\ 0& P_n\end{pmatrix}+\begin{pmatrix} B_2^{(n)*}B_2^{(n)}+P_n-C_n & 0\\ 0& 0 \end{pmatrix} +\begin{pmatrix} \sum_{i=3}^m B_i^{(n)*}B_i^{(n)} & 0\\ 0 & 0 \end{pmatrix}\\
    &=\begin{pmatrix} P_n & 0\\ 0 &P_n\end{pmatrix}=I_{P_n\H\oplus P_n\H}.
\end{align*}
So, $(A_1^{(n)},...,A_m^{(n)})$ satisfies the column isometry identity in $B(P_n\H\oplus P_n\H)\simeq \mathbb{M}_{2n}$.
For each $n\geq 1$, let $\pi_n: C^*(I_{\mathcal{H}},A_1,...,A_m)\to C^*(P_n\oplus P_n, A_1^{(n)},...,A_m^{(n)})\subset B(P_n\H\oplus P_n\H)$ be the representation induced by $A_i\mapsto A_i^{(n)}$ for each $1\leq i\leq m$. It remains to show that $\pi=\oplus_n\pi_n$ is an isometry. 

We now consider $P_n \H \oplus P_n \H$ as a subspace of $\H \oplus \H$ in the obvious way, thereby viewing each $A_i^{(n)}$ as an operator on $\H \oplus \H$ (acting as zero on $(P_n \H \oplus P_n \H)^\perp$).  Since the $P_n$ converge *-strongly to $I_{\mathcal{H}}$, 
we have that $A_1^{(n)}$ converges *-strongly to $A_1\oplus 1_\H$ and  $A_i^{(n)}$ converges *-strongly to $A_i\oplus 0$ for $2\leq i\leq m$. So, for any noncommutative *-polynomial $q$ in two variables, $$q(A_1^{(n)},...,A_m^{(n)})\xrightarrow{S^*OT} q(A_1\oplus I_{\mathcal{H}},A_2\oplus 0,...,A_m\oplus 0).$$
Note that $(A_1\oplus I_{\mathcal{H}}, A_2\oplus 0,...,A_m\oplus 0)$ is also a universal $m$-column isometry. Indeed, it clearly satisfies the column isometry identity, and, since $$\|q(A_1\oplus I_{\mathcal{H}},A_2\oplus 0,...,A_m\oplus 0)\|=\|q(A_1,...,A_m)\oplus q(I_{\mathcal{H}},0,...,0)\|\geq \|q(A_1,...,A_m)\|$$ for any noncommutative *-polynomial $q$ in two variables, the natural surjection induced by mapping $A_1\mapsto A_1^{(n)}\oplus 1_\H$ and $A_i\mapsto A_i^{(n)}\oplus 0$ for $2\leq i\leq m$ is isometric.

Now, let $q$ be a noncommutative *-polynomial in $m$ variables and $\epsilon>0$. Assume $\|q(A_1,...,A_m)\|=1$. Then  $\|q(A_1\oplus I_{\mathcal{H}},A_2\oplus 0,...,A_m\oplus 0)\|=1$ and $$\|q(A_1^{(n)},...,A_m^{(n)})\|\geq 1-\epsilon,$$ for all sufficiently large $n$. Hence, 
$$\|\pi(q(A_1,...,A_m))\|\geq \|\pi_n(q(A_1,...,A_m))\|=\|q(A_1^{(n)},...,A_m^{(n)})\|\geq 1-\epsilon$$ for all sufficiently large $n$. Since $\epsilon$ was arbitrary, $\pi$ is isometric on the set of noncommutative *-polynomials on $A_1,...,A_m$, which, by continuity, means it is isometric on $C^*(I_{\mathcal{H}},A_1,...,A_m)\simeq \mathcal{P}_m$. 
\end{proof}

\end{document}